\documentclass[12pt]{amsart}
\usepackage{amssymb,amsmath,amsthm,mathrsfs,MnSymbol,xcolor}
\usepackage[all,cmtip]{xy}
\usepackage{tikz-cd}
\usepackage{arydshln,caption}   
\usepackage[inline]{enumitem}
\usepackage{hyperref}
\usepackage[msc-links]{amsrefs}
\hypersetup{
  colorlinks   = true,	
  urlcolor     = blue,
  linkcolor    = blue,
  citecolor   = red
}
\calclayout
\allowdisplaybreaks
\newtheorem{theorem}{Theorem}
\newtheorem{lemma}[theorem]{Lemma}
\newtheorem{remark}[theorem]{Remark}
\newtheorem{corollary}[theorem]{Corollary}

\newtheorem{proposition}[theorem]{Proposition}
\numberwithin{theorem}{section} \numberwithin{equation}{section}
\newcommand{\beq}{\begin{small} \begin{equation}}
\newcommand{\eeq}{\end{equation} \end{small}}
\newcommand{\beqn}{\begin{small} \begin{equation*}}
\newcommand{\eeqn}{\end{equation*} \end{small}}
\DeclareMathAlphabet{\mathpzc}{OT1}{pzc}{m}{it}
\begin{document}
\title{On projective K3 surfaces $\mathcal{X}$ with $\mathrm{Aut}(\mathcal{X})=(\mathbb{Z}/2\mathbb{Z})^2$}
\author{Adrian Clingher}
\address{Department of Mathematics and Statistics, University of Missouri - St. Louis, St. Louis, MO 63121, United States}
\email{clinghera@umsl.edu}
\author{Andreas Malmendier}
\address{Department of Mathematics \& Statistics, Utah State University, Logan, UT 84322, United States}
\email{andreas.malmendier@usu.edu}
\author{Xavier Roulleau}
\address{Universit\'e d'Angers, CNRS, Larema, SFR MathStic, F-49000 Angers, France}
\email{xavier.roulleau@univ-angers.fr}
\begin{abstract}
We prove that every K3 surface with automorphism group $(\mathbb{Z}/2\mathbb{Z})^2$ admits an explicit birational model as a double sextic surface. This model is canonical for Picard number greater than 10. For Picard number greater than 9, the K3 surfaces in question possess a second birational model, in the form of a projective quartic hypersurface, generalizing the Inose quartic.
\end{abstract}
\maketitle
\section{Introduction}
Let $\mathcal{X}$ be a smooth complex projective K3 surface. We denote the N\'eron-Severi lattice of $\mathcal{X}$ by $\mathrm{NS}(\mathcal{X})$. This lattice is even and has signature $(1, \rho_\mathcal{X}-1)$, where $\rho_\mathcal{X}$ is the Picard number of $\mathcal{X}$. The value of $\rho_\mathcal{X}$ ranges from 1 to 20. A lattice polarization on $\mathcal{X}$ is defined as a primitive lattice embedding $i\colon L \hookrightarrow \mathrm{NS}(\mathcal{X})$, where $i(L)$ contains a pseudo-ample class. Here, $L$ is an even indefinite lattice of rank $\rho_L$ and signature $(1,\rho_L-1)$, with $1 \leq \rho_L \leq 20$. Two K3 surfaces $(\mathcal{X},i)$ and $(\mathcal{X}',i')$ are said to be isomorphic under $L$-polarization if there exists an analytic isomorphism $\alpha \colon \mathcal{X} \rightarrow \mathcal{X}'$ and a lattice isometry $\beta \in O(L)$ such that $\alpha^* \circ i' = i \circ \beta$, where $\alpha^*$ is the induced morphism at the cohomology level.\footnote{Our definition of isomorphic lattice polarizations coincides with the one used by Vinberg, and is slightly more general than the one used in \cite{MR1420220}*{Sec.~1}.} $L$-polarized K3 surfaces are classified up to isomorphism by a coarse moduli space, which is a quasi-projective variety of dimension $20-\rho_L$. A \emph{general} $L$-polarized K3 surface $(\mathcal{X},i)$ satisfies $i(L)=\mathrm{NS}(\mathcal{X})$.
\par It is known that the automorphism group $\mathrm{Aut}(\mathcal{X})$ of a K3 surface $\mathcal{X}$ is isomorphic, up to a finite group, to the factor group $O(\mathrm{NS}(\mathcal{X}))/\mathcal{W}_\mathcal{X}$, where $\mathcal{W}_\mathcal{X}$ is the subgroup generated by reflections associated with elements of $\mathrm{NS}(\mathcal{X})$ of square $(-2)$. Nikulin \cites{MR633160, MR556762, MR752938} and Vinberg \cite{MR2429266} classified K3 surfaces $\mathcal{X}$ with finite automorphism group by their N\'eron-Severi lattices $\mathrm{NS}(\mathcal{X})$ for $\rho_\mathcal{X} \ge 3$, resulting in a list of 118 lattices. This paper focuses on the case $\mathrm{Aut}(\mathcal{X}) \cong (\mathbb{Z}/2\mathbb{Z})^2$. According to Kondo \cite{MR1029967}, in such a situation the lattice $\mathrm{NS}(\mathcal{X})$ is 2-elementary. In fact, all possible N\'eron-Severi lattices may be listed as follows:
\beq
\label{eqn:lattices}
\begin{split}
 H \oplus E_8  \oplus E_8, \, H \oplus E_8  \oplus E_7, \,  H \oplus E_8  \oplus D_6,  \,  H \oplus E_8  \oplus D_4 \oplus A_1, \,  H \oplus E_8  \oplus A_1^{\oplus 4}, \, H \oplus D_8  \oplus D_4,\\
 H \oplus E_7  \oplus A_1^{\oplus 4}, \, H \oplus D_6  \oplus A_1^{\oplus 4}, \, H \oplus D_4  \oplus A_1^{\oplus 5}, \, H\oplus N \cong H(2) \oplus D_4^{\oplus 2}, \, H \oplus A_1^{\oplus 8}, \, H(2) \oplus A_1^{\oplus 7}.
\end{split}
\eeq
{\small Above, we utilize the following notations: $L_1 \oplus L_2$ refers to the orthogonal direct sum of the two lattices $L_1$ and $L_2$, $L(\lambda)$ is obtained by multiplying the lattice pairing of $L$ with $\lambda \in \mathbb{Z}$, $\langle R \rangle$ denotes a lattice with matrix $R$ in some basis. The lattices $A_n$, $D_m$, and $E_k$ are the negative definite root lattices associated with the respective root systems. $H$ is the unique even unimodular hyperbolic rank-two lattice, and $N$ is the negative definite rank-eight Nikulin lattice, as defined, for instance, in \cite{MR728142}*{Sec.~5}.
\par Given a lattice $L$, we denote its discriminant group as $L^\vee/L$, and its associated discriminant form as $q_L$. A lattice $L$ is termed \emph{2-elementary} if $D(L) \cong (\mathbb{Z}/2\mathbb{Z})^\ell$, where $\ell=\ell_L$ is the \emph{length} of $L$. A result of Nikulin \cite{MR633160b}*{Thm.~4.3.2} asserts that hyperbolic, even, 2-elementary, primitive sublattices $L$ of the even, unimodular K3 lattice of signature $(3, 19)$ are uniquely determined by their rank $\rho_L$, length $\ell_L$, and a third invariant $\delta_L \in \{ 0, 1\}$, called \emph{parity}.  Here, $\delta_L = 0$ if $q_L(x)$ takes values in $\mathbb{Z}/2\mathbb{Z} \subset \mathbb{Q}/2\mathbb{Z}$ for all $x \in D(L)$, and $ \delta_L=1$ otherwise.}
\medskip
\par The present article builds upon previous work by the authors, as well as others. In \cite{Roulleau22}, the third author studied the geometry of K3 surfaces with finite automorphism group. In particular, the configuration of all rational curves was determined in each case. The first two authors studied the same families of K3 surfaces from the point of view of elliptic fibrations \cite{Clingher:2022}. Detailed classifications were given, including a description of coarse moduli spaces, as well as explicit birational models. Other examples of explicit elliptic fibrations on K3 surfaces of 2-elementary type were given by Garbagnati and Salgado in~\cites{MR4130832} and  Balestrieri et al.~\!in~\cites{MR3882710}. 
\par  In this article, we establish explicit birational models for all general $L$-polarized K3 surfaces, with $L$ a member of the list~(\ref{eqn:lattices}). Specifically, we construct normal forms, as quartic hypersurfaces in $\mathbb{P}^3$, as well as double sextic surfaces. The automorphisms in $\mathrm{Aut}(\mathcal{X})$ can then be seen as explicit algebraic transformations in the double sextic context. For  $\rho_L > 10$, the double sextic structure is canonical. The linear systems corresponding to the two projective models are also described in detail.
\par The article is structured as follows: Section~\ref{sec:statment} outlines the main results, namely Theorems~\ref{thm1_intro} and~\ref{thm2_intro}. The construction of the relevant birational models as well as the proof of the theorems are then divided up into different cases, corresponding to their Picard number. Section~\ref{sec:rank9} presents the treatment of K3 surfaces of Picard number 9. Similarly, Sections~\ref{sec:rank10} and ~\ref{sec:rank11ff} discuss K3 surfaces with Picard number 10 and with Picard number greater than 10, respectively.
\section*{Acknowledgement}
The authors would like to thank the referees for their careful reading of the article and the detailed comments and corrections.
A.M. acknowledges support from the Simons Foundation through grant no.~202367.
X.R. acknowledges support from Centre Henri Lebesgue ANR-11-LABX-0020-01.
\section{Statement of results}
\label{sec:statment}
A central role in this paper is played by the concept of a \emph{double sextic structure} on a K3 surface  \cites{MR595204, MR805337,MR2030225}. Let $S \subset \mathbb{P}^2 $ be a reduced plane curve of degree six; we say that the sextic is {\it admissible} if  the double cover $\pi_\mathcal{S}\colon \mathcal{S}  \rightarrow \mathbb{P}^2$ branched over $S$ is a surface with, at worst, rational double point singularities. The singularities of such a sextic  are called \emph{simple singularities} in \cite{MR661198} and general results of Brieskorn then guarantee that the minimal desingularization $\mathcal{X}$ of $\mathcal{S}$, is a K3 surface. The above process defines a double cover morphism $\mathcal{X} \rightarrow \mathbb{P}^2 $, which will be referred to as a \emph{double sextic structure} on $\mathcal{X}$.
\par We note that the K3 surface $\mathcal{X}$, as well as the double sextic structure $\mathcal{X} \rightarrow \mathbb{P}^2 $, may alternatively be obtained by first performing a sequence of blow-ups on $\mathbb{P}^2$, until the singularities of $S$ are resolved, followed by taking a double cover of the resulting rational surface, branched over an appropriate smooth branch locus. 
\par We also note that a double sextic structure $\mathcal{X} \rightarrow \mathbb{P}^2 $ induces a natural antisymplectic involution $k_\mathcal{X} \colon \mathcal{X} \rightarrow \mathcal{X}$, as well as a special base-point-free linear system $\vert \mathpzc{D}_2 \vert $, with $\mathpzc{D}_2$ quasi-ample, satisfying $\mathpzc{D}_2^2 = 2$ and $h^0(\mathpzc{D}_2) = 3$.  The linear system $\vert \mathpzc{D}_2 \vert $ completely determines the double sextic structure $\mathcal{X} \rightarrow \mathbb{P}^2$.  
\par In the generic situation, the sextic curve  $S$ is non-singular, and in general position, which leads to $\rho_\mathcal{X} =1$. However, the Picard number of $\mathcal{X}$ may increase, as the curve $S$ may lie in a special position or may carry singularities. Multiple such examples were constructed in \cites{MR805337}. The minimal resolutions of the double sextic surfaces we derive in this article have Picard numbers between 9 and 18.  
\par A classical geometric situation in which double sextic structures occur is the case when $\mathcal{X}$ admits a singular \emph{quartic normal form}, i.e., $\mathcal{X}$ may be obtained as the minimal resolution of a quartic surface $\mathcal{K}$ in $\mathbb{P}^3$ with rational double point singularities. Assume that $\mathrm{p} \in \mathcal{K}$ is such a rational double point. Considering the lines in $\mathbb{P}^3$ passing through $\mathrm{p}$, one obtains a projection morphism,  associated with the projection with center $\mathrm{p}$, 
\begin{equation}
\label{dsextic} 
   \mathcal{K} - \{ \mathrm{p} \} \ \rightarrow \ \mathbb{P}^2 ,
 \end{equation} 
which generically is a double cover.  Upon performing the desingularization of $\mathcal{K}$, the morphism (\ref{dsextic}) extends to a double sextic structure $\mathcal{X} \rightarrow \mathbb{P}^2$.   

\par In this article, we shall primarily focus on two types of double sextic structures: on the one hand, we construct double sextics $\mathcal{S}$ where the branch locus in $\mathbb{P}^2 = \mathbb{P}(u, v, w)$ is a bi-quadratic sextic that is generically irreducible. Such a double sextic is given by
\beq
\label{eqn:double-sextics_general_intro}
   \mathcal{S}\colon \quad  y^2 =   c_0 w^6 + c_2(u, v) w^4 + c_4(u, v) w^2  + c_6(u, v) \, ,
\eeq
where $c_n$ are homogeneous polynomials of degree $n$. In this context, the automorphism group $\mathrm{Aut}(\mathcal{X})$ of the associated K3 surface $\mathcal{X}$ contains several involutions by construction: 
\begin{enumerate*}[label=(\roman*)] 
\item An antisymplectic involution $\jmath_\mathcal{X}$ is induced by the involution $\jmath_\mathcal{S} \colon w \mapsto -w$ on the double sextic whose invariant locus contains a smooth curve of genus two which is the hyperelliptic curve $y^2 = c_6(u, v)$.
\item A commuting antisymplectic involution is the aforementioned involution $k_\mathcal{X}$ that is induced by the involution $k_\mathcal{S} \colon  y \mapsto - y$. 
\item Moreover, the composition of the two anti-symplectic involutions induces the symplectic involution $\imath_\mathcal{S}= \jmath_\mathcal{S} \circ k_\mathcal{S}\colon (w, y) \mapsto (-w, -y)$ on $\mathcal{S}$ and $\imath_\mathcal{X}= \jmath_\mathcal{X} \circ k_\mathcal{X}$ on $\mathcal{X}$, respectively.
\end{enumerate*}
\par On the other hand, our investigation will encompass the desingularization of double covers branched over a reducible sextic in the projective plane, which is expressed as the union of a conic and a plane quartic curve. Such a double sextic is given by
\beq
\label{eqn:double-sextics_prime_general_intro}
\mathcal{S}'\colon \quad \tilde{y}^2 = C(u, v, w) \cdot Q(u, v, w) \, ,
\eeq
where $C$ and $Q$ are homogeneous polynomials of degree 2 and 4, respectively. 
\par We will demonstrate that every general $L$-polarized K3 surface, where $L$ is one of the lattices in Equation~(\ref{eqn:lattices}),  admits birational models as  double sextics of the form $\mathcal{S}$ and $\mathcal{S'}$.\footnote{For ranks $\rho_L=9, 10$ there is only the birational model $\mathcal{S}$, and for $L=H \oplus N$ the sextic $\mathcal{S}$ has a different form, given by Equation~(\ref{eqn:double_sextic_rank10}).}  Our findings are as follows:
\begin{theorem}
\label{thm1_intro}
Let $\mathcal{X}$ be a K3 surface with $\mathrm{Aut}(\mathcal{X}) \cong (\mathbb{Z}/2\mathbb{Z})^2$ so that $\mathrm{NS}(\mathcal{X}) \cong L$, for $L$ in Equation~(\ref{eqn:lattices}). There exists a double sextic $\mathcal{S}$ whose minimal resolution is isomorphic to $\mathcal{X}$ and has the following properties:
\begin{enumerate}
\item For $L \not = H \oplus N$ there are polynomials $c_n$ of degree $n=0, 2, 4, 6$ so that $\mathcal{S}$ is given by~(\ref{eqn:double-sextics_general_intro}). For $L  = H \oplus N$ the double sextic is of the form
\beqn
\label{eqn:double_sextic_intro}
  \mathcal{S}\colon \quad y^2   = w \big( c_3(u, v)  \, w^2 + a_4(u, v) \, w + d_5(u, v)  \big) \, ,
\eeqn
where $c_3, a_4, d_5$ are homogeneous polynomials of degree 3, 4, 5, respectively.
\item Let $S$ be the sextic branch curve. For $L \not = H \oplus N , H \oplus D_8 \oplus D_4$ the double sextic $\mathcal{S}$ is irreducible, and $S$ has one singularity of type $\mathbf{A}_{2n+1}$ for $\rho_L =10+n$, $0 \le n \le 8$, and no singularity for $\rho_L=9$. For $L = H \oplus N, H \oplus D_8 \oplus D_4$ the sextic curve $S$ is the union of a line and an irreducible quintic.
\item There is a  divisor $\mathpzc{D}_2 \in \mathrm{NS}(\mathcal{X})$, which is  nef,  of square-two, and invariant under the action of $\imath_\mathcal{X}$. For the linear system $\vert \mathpzc{D}_2 \vert$ the corresponding rational map is the projection onto $\mathbb{P}^2=\mathbb{P}(u, v, w)$ in Equation~(\ref{eqn:double-sextics_general_intro}).
\end{enumerate}
If $\rho_L > 10$ or $L = H \oplus A_1^{\oplus 8}$, then there are polynomials $C$ and $Q$ such that the minimal resolution of  $\mathcal{S}'$ in~(\ref{eqn:double-sextics_prime_general_intro}) is isomorphic to $\mathcal{X}$.
\end{theorem}
As an immediate consequence we observe that the non-trivial elements of $\mathrm{Aut}(\mathcal{X})$ in Theorem~\ref{thm1_intro} are the involutions $\imath_\mathcal{S}, \jmath_\mathcal{S}, k_\mathcal{S}$ with $\imath_\mathcal{S}= \jmath_\mathcal{S} \circ k_\mathcal{S}$ on the double sextic $\mathcal{S}$. We make the following:
\begin{remark}
\label{rem:branch_locus}
In general, the determination of the fixed locus of the antisymplectic involution $\jmath_\mathcal{X}$ was accomplished by Nikulin \cite{MR633160b}*{Thm.4.2.2}. By applying the result to the lattices in Equation (\ref{eqn:lattices}) --excluding the case $L = H \oplus N$-- it follows that the fixed locus is comprised of a smooth, irreducible curve of genus $g=2$ and $k$ rational curves, with $2 \leq k \leq 9$, such that $k=(\rho_L -\ell_L)/2$. For $L = H \oplus N$ the fixed locus consists of a genus-3 curve and two rational curves.
\end{remark}
\par Furthermore, the birational model $\mathcal{S}$ turns out to be canonical for $\rho_L >10$, in a sense that we explain below. To see this, let us first recall the following facts: We define a \emph{Jacobian elliptic fibration} on $\mathcal{X}$ to be a pair $(\pi_\mathcal{X},\sigma_\mathcal{X})$ consisting of a proper map of analytic spaces $\pi_\mathcal{X} \colon \mathcal{X}  \rightarrow \mathbb{P}^1 =\mathbb{P}(u, v) $, whose general fiber is a smooth curve of genus one, and a section $\sigma_\mathcal{X} \colon \mathbb{P}^1 \to \mathcal{X}$ in the elliptic fibration $\pi_\mathcal{X}$. The group of sections of the Jacobian fibration is the \emph{Mordell-Weil group} $\operatorname{MW}(\mathcal{X}, \pi_\mathcal{X})$.  By a slight abuse of notation, we will not distinguish between the smooth elliptic K3 surface and its singular Weierstrass model. The latter is obtained from the former by blowing down the components of the reducible fibers which do not meet the section.
\par  In \cite{MR2274533} van Geemen and Sarti studied a certain family of Jacobian elliptic K3 surfaces. The family is given by
\beq
\label{eqn:vgs_intro}
  \mathcal{X}\colon \quad Y^2 Z  = X \big( X^2 + a_4(u, v)  X Z + b_8(u, v) Z^2  \big) \, ,
\eeq
where $a_4$ and $b_8$ are homogeneous polynomials of degree four and eight, respectively.  This Jacobian elliptic fibration always admits a section of order two, or, \emph{2-torsion section}.  The 2-torsion section is $(0:0:1)$ if the zero section is chosen as $(0:1:0)$. The fiberwise translation by the 2-torsion section acts as a Nikulin involution $\imath_\mathcal{X}$ on the elliptic fibration~(\ref{eqn:vgs_intro}). It is called \emph{van Geemen-Sarti involution}. The K3 surface $\mathcal{X}$ also admits the involution acting as the (fiberwise) hyperelliptic involution $\jmath_\mathcal{X}$, and a third involution obtained by composition $k_\mathcal{X} = \imath_\mathcal{X} \circ \jmath_\mathcal{X}$. 
\par It turns out that the general K3 surface $\mathcal{X}$ in~(\ref{eqn:vgs_intro}) has $\mathrm{NS}(\mathcal{X}) \cong  H \oplus N$. Conversely, it is known that every general $H \oplus N$-polarized K3 surface admits a \emph{unique} fibration given by Equation~(\ref{eqn:vgs_intro}); see  \cite{MR2274533}. A generalization of the statement for $L$ in Equation~(\ref{eqn:lattices}) with $\rho_L >10$ was proved by the first two authors in \cite{Clingher:2022}: in fact, any $\mathcal{X}$ with $\mathrm{NS}(\mathcal{X})\cong L$, with $L$ in (\ref{eqn:lattices}) and $\rho_L>10$, admits an $H\oplus N$-polarization and thus admits an Equation of type (\ref{eqn:vgs_intro}). As for nomenclature, we shall refer to a Jacobian elliptic fibration associated with an $H$-embedding into a lattice $L \cong H \oplus K$ where $K$ is a negative-definite lattice of ADE-type as  \emph{standard}.  An $H$-embedding on $\mathcal{X}$ determines \cite{MR2369941}*{Lemma~3.6} an elliptic fibration $\mathcal{X} \rightarrow \mathbb{P}^1 $, together with a choice of section in this fibration.  The standard Jacobian elliptic fibration will not be used in the remainder of the article. A Weierstrass equation for the standard fibration (which is different from Equation~(\ref{eqn:vgs_intro}) and does not have a 2-torsion section) and a detailed analysis of its singular fibers was given in \cite{Clingher:2022}. In contrast, the elliptic fibration associated with Equation~(\ref{eqn:vgs_intro}) carrying an underlying van Geemen-Sarti involution will be referred to as \emph{alternate} fibration.  We will then prove the following:
\begin{proposition}
\label{prop:canonical_intro}
Let $\mathcal{X}$ be a K3 surface with $\mathrm{NS}(\mathcal{X}) \cong L$, for an $L$ in Equation~(\ref{eqn:lattices}) with  $\rho_L > 10$. Then $\mathcal{X}$ is isomorphic to the minimal resolution of the double sextic $\mathcal{S}$ in Equation~(\ref{eqn:double-sextics_general_intro}) where $c_0=0$ and $c_2, c_4, c_6$ are uniquely determined by the alternate fibration, and $\imath_\mathcal{S}, \jmath_\mathcal{S}$ induce the van Geemen-Sarti and hyperelliptic involution.
\end{proposition}
Based on Proposition~\ref{prop:canonical_intro} we call the equation for the double sextic $\mathcal{S}$ for $\rho_L > 10$ \emph{canonical}.
\par Up to this point, our focus has primarily been on the double sextic surface $\mathcal{S}$. However, the general $L$-polarized K3 surface $\mathcal{X}$, with $L$ in Equation~(\ref{eqn:lattices}) and $\rho_L > 10$ or $L = H \oplus A_1^{\oplus 8}$, also possesses a second birational model. This model is that of the double sextic $\mathcal{S}'$ in Equation~(\ref{eqn:double-sextics_prime_general_intro}). Within the context of $\mathcal{S}'$, there is also a quartic projective hypersurface $\mathcal{K} \subset \mathbb{P}^3 = \mathbb{P}(u, v, w, y)$ or simply quartic, given by
\beq
\label{eqn:quartic_general_intro}
\mathcal{K}\colon \quad y^2 C(u, v, w) = Q(u, v, w) \, ,
\eeq
where $C$ and $Q$ are the same polynomials as those used in Equation~(\ref{eqn:double-sextics_prime_general_intro}) to define $\mathcal{S}'$. It follows that for every quartic $\mathcal{K}$, we obtain an associated double sextic $\mathcal{S}'$ as defined by Equation~(\ref{eqn:double-sextics_prime_general_intro}), and vice versa, where the relation between Equation~(\ref{eqn:double-sextics_prime_general_intro}) and Equation~(\ref{eqn:quartic_general_intro})  is $\tilde{y}= C(u, v, w) \, y$.
\par The significance of this quartic can be seen as follows. In the works \cites{MR2369941,MR2279280}, it was demonstrated that a K3 surface $\mathcal{X}$ over the complex numbers with N\'eron-Severi lattice $H \oplus E_8 \oplus E_8$ possesses a birational model isomorphic to the quartic $\mathcal{K}_{\mathrm{Inose}}$ in $\mathbb{P}^3=\mathbb{P}(\mathbf{X}, \mathbf{Y}, \mathbf{Z}, \mathbf{W})$, as defined by the equation
\beqn
\mathcal{K}_{\mathrm{Inose}}\colon \quad   0= \  \mathbf{Y}^2 \mathbf{Z} \mathbf{W}-4 \mathbf{X}^3 \mathbf{Z}+3 \alpha \mathbf{X} \mathbf{Z} \mathbf{W}^2+ \beta \mathbf{Z} \mathbf{W}^3 
-   \frac{1}{2} \big( \mathbf{Z}^2  \mathbf{W}^2 +  \mathbf{W}^4 \big) \, ,
\eeqn
where $\alpha$ and $\beta$ are two complex parameters. This two-parameter family was initially introduced by Inose in \cite{MR578868} and is called \emph{Inose quartic}. Notably, the quartic $\mathcal{K}$ defined in Equation~(\ref{eqn:quartic_general_intro}) coincides with the Inose quartic when $\rho_L=18$. Furthermore, it provides multi-parameter generalizations of the Inose quartic for Picard numbers $10 \le \rho_L <18$. We will prove the following:
\begin{theorem}
\label{thm2_intro}
Let $\mathcal{X}$ be a K3 surface with $\mathrm{NS}(\mathcal{X}) = L$, for  $L = H \oplus A_1^{\oplus 7}$ or $L$ in Equation~(\ref{eqn:lattices}) with $L \not = H(2) \oplus A_1^{\oplus 7}, H \oplus N$.  There exists a quartic $\mathcal{K}$ whose minimal resolution is isomorphic to $\mathcal{X}$ and has the following properties:
\begin{enumerate}
\item There are polynomials $C, Q$ so that $\mathcal{K}$  is given by~(\ref{eqn:quartic_general_intro}).
\item $\mathcal{K}$ has one rational double-point at $\mathrm{p}_1\colon \,  [u: v: w : y]= [0: 0: 0: 1]$ for $\rho_L=9$, and for $\rho_L \ge 10$  two rational double-point singularities at
\beqn
\label{eqn:singular_points_intro}
   \mathrm{p}_1\colon \ [u: v: w : y]= [0: 0: 0: 1] \,, \qquad   \mathrm{p}_2\colon \ [u: v: w : y]= [0 : 0: 1: 0] \,.
\eeqn  
\item For $\rho_L \ge 10$ there is an explicit birational map between $\mathcal{K}$ and $\mathcal{S}$, such that the polarization divisor for $\mathcal{K}$ is $\mathpzc{H} = \mathpzc{D}_2 + \mathrm{b}_1 + \dots + \mathrm{b}_N$ where $\mathrm{b}_i$ are the exceptional curves of the singularity at $\mathrm{p}_2$, and  $\mathpzc{D}_2$ is the same as in Theorem~\ref{thm1_intro}.
\end{enumerate}
\end{theorem}
We also make the following:
\begin{remark}
The projection of  $\mathcal{K}$ onto $\mathbb{P}^2$ with center either $\mathrm{p}_1$ or  $\mathrm{p}_2$  recovers the two double sextic $\mathcal{S}'$ or $\mathcal{S}$, respectively, for $\rho_L \ge 10$.
\end{remark}
\begin{remark}
We also provide explicit expressions for the divisors in Theorem~\ref{thm1_intro} and~\ref{thm2_intro} in terms of the dual graph of rational curves for the associated $L$-polarized K3 surfaces.
\end{remark}
\section{K3 surfaces of Picard number 9}
\label{sec:rank9}
In Picard number 9, we consider two lattice polarizations, namely $H(2) \oplus A_1^{\oplus 7}$ and  $H \oplus A_1^{\oplus 7}$. Only the general K3 surface $\mathcal{X}$ polarized by the former lattice satisfies $\mathrm{Aut}(\mathcal{X}) \cong (\mathbb{Z}/2)^2$. In contrast, the general $H \oplus A_1^{\oplus 7}$-polarized K3 surface $\mathcal{X}$ has $\mathrm{Aut}(\mathcal{X}) \cong \mathbb{Z}/2$.  However, considerations regarding an $H \oplus A_1^{\oplus 7}$-polarization will turn out to be useful for constructions in higher Picard number.
\subsection{\texorpdfstring{$H(2) \oplus A_1^{\oplus 7}$-polarized K3 surfaces}{Families polarized by H(2) + 7 A1}}
It is known that for at most eight points on $\mathbb{P}^2$ in general position, there always exists an irreducible sextic curve containing these points as nodes. 
\par Let us examine the situation where a sextic curve $S'$ in $\mathbb{P}^2$ has the eight ordinary double points $\mathscr{P} = \{ s_1, \dots, s_8\}$ in general position in $\mathbb{P}^2$. In this situation, we use the construction described in detail in \cite{MR4001046}, to obtain a rational map to weighted projective space: The space of plane cubics through $\mathscr{P}$ has dimension 2. Let $\{U, V \}$ be a basis for the pencil.  The space of plane sextics with double points at $\mathscr{P}$ has dimension 4 and is generated by $\{ U^2, UV , V^2, W\}$ for some sextic $W$.  Further, the space of plane nonics with 3-fold points at $\mathscr{P}$ has dimension 7 and is generated $\{ U^3 , U^2V, UV^2, V^3, UW, VW, R\}$ for some nonic $R$.  
\par For a chosen basis of nonics with 3-fold points at $\mathscr{P}$, one obtains the rational map
\beq
\begin{split}
 \psi \colon \quad \mathbb{P}^2  & \ \dasharrow  \ \mathbb{P}(1,1,2,3)\, , \\
 [x : y : z]  &\ \mapsto \ \big[ U(x, y, z) \ : \ V(x, y, z) \ : \ W(x, y, z) \ : \ R(x, y, z) \big] \, .
\end{split} 
\eeq
The closure of the image of $\psi$ will be denoted by $\mathcal{Z} = \overline{\mathrm{Im}(\psi)}$.  By construction, it is a degree-one del Pezzo surface $\mathcal{Z}$, i.e., the blow up of $\mathbb{P}^2$ in $\mathscr{P}$; see \cite{MR4001046}*{Sec.~3.1}. We have the following:
\begin{lemma}
\label{lem:Z}
The degree-one del Pezzo surface $\mathcal{Z}$ in $\mathbb{P}(1, 1, 2, 3) = \mathbb{P}(U, V, W, R)$ is given by the equation
\beq
\label{eqn:WP}
\mathcal{Z}\colon \quad R^2 = c_0 W^3 + c_2(U, V) W^2 + c_4(U, V) W + c_6(U, V) \, ,
\eeq
where $c_2, c_4, c_6$ are polynomials of degree 2, 4, and 6, respectively, and $c_0 \in \mathbb{C}^\times$. In particular, $\mathcal{Z}$ does not pass through the singular points $[0:0:1:0]$ or $[0:0:0:1]$, i.e., the singular points of $\mathbb{P}(1,1,2,3)$.
\end{lemma}
\begin{proof}
The anti-canonical ring of a del Pezzo surface of degree one is generated by its elements of degree at most three. By computing the dimension of $H^0( \mathcal{O}(-n K_\mathcal{Z}))$ for $n= 1, 2, 3$ it follows that the anti-canonical ring is generated by  $H^0( \mathcal{O}(- K_\mathcal{Z}))$ and one element each of $H^0( \mathcal{O}(-n K_\mathcal{Z}))$ for $n= 2, 3$\ \cite{MR2062787}*{Sec.~3.5}.  Thus, the image of $\psi$ is cut out by a single equation of degree six.
\end{proof}
Closely related to $\mathcal{Z}$ is a rational elliptic surface $\widetilde{\mathcal{Z}}  \rightarrow \mathbb{P}^1 =\mathbb{P}(u, v)$ with section, given by
\beq
\label{eqn:srfc_Z_RES}
\widetilde{\mathcal{Z}}\colon \quad Y^2 Z = c_0 X^3 + c_2(u, v) X^2 Z + c_4(u, v)  X Z^2 + c_6(u, v) Z^3 \, ,
\eeq
whose general fiber in $\mathbb{P}^2 = \mathbb{P}(X, Y, Z)$  is a smooth curve of genus one. It is obtained by blowing up $\mathcal{Z}$ one time, namely at the ninth base point of $\{U, V\}$. Conversely, if one blows down the section of $\widetilde{\mathcal{Z}}$, given by $O: [X: Y: Z] = [0:1:0]$, one recovers $\mathcal{Z}$ using
\beq
 \big( [u:v], \, [X:Y:Z] \big) \ \mapsto \  [ U : V : W : R ] = [ uZ : vZ: XZ : YZ^2 ] \,.
\eeq
In our situation, the remaining exceptional divisors are not contained in the fibers, but yield independent sections of the fibration and generate a lattice of type $E_8$. Having chosen $O$ as section of the elliptic fibration, the N\'eron-Severi lattice $\mathrm{NS}(\mathcal{Z})$ is generated by the pull-back $\mathcal{L}'$ of a line and the 8 exceptional divisors $ \mathrm{E}_i= [\sigma_i]$, $i=1, \dots, 8$ that are obtained from the classes of the sections.  The anti-canonical divisor of $\mathcal{Z}$ is given by
 \beq
  -K_\mathcal{Z} = 3 \mathcal{L}' -  \mathrm{E}_1 - \dots -  \mathrm{E}_8 \,.
 \eeq 
\par The strict transform of a fixed nodal sextic $S'$ in $\mathcal{Z}$ is a smooth genus-2 curve $B$.  Without loss of generality, we can assume that this sextic is given by $W=0$. The genus-2 curve inside $\mathcal{Z}$ then is the hyperelliptic curve $B \colon R^2 = c_6(U, V)$. Note that any other choice of genus-2 curve is obtained by  a shift $W \mapsto f_0 W + f_2(U, V)$ where $f_0 \in \mathbb{C}^\times$ and $f_2$ is a homogeneous polynomial of degree 2. Under such shift, coefficients in Equation~(\ref{eqn:WP}) will change according to
\beqn
 c_0 \mapsto c_0 f_0^3, \ c_2 \mapsto (c_2  + 3 c_0 f_2) f_0^2, \ c_4 \mapsto (c_4 + 2 c_2 f_2 + 3 c_0 f_2^2) f_0, \ c_6 \mapsto c_6 + c_4 f_2 + c_2 f_2^2 + c_0 f_2^3 \,
\eeqn
as the genus-2 curve $B$ moves in the 4-dimensional family $f_0 W + f_2(U, V) =0$.
\par In the situation above, it is known that the minimal resolution of the double cover of $\mathcal{Z}$ branched over the nodal sextic $S'$ is a K3 surface $\mathcal{X}$. Explicitly, from the double cover
\beq
\label{eqn:map_rank9}
 \varphi\colon \ \mathcal{S} \subset \mathbb{P}(1, 1, 1, 3)  \rightarrow  \mathcal{Z}  \subset \mathbb{P}(1, 1, 2, 3) , \ [u: v : w : y ]  \mapsto  [ U : V : W : R ]  = [ u : v: w^2: y] ,
\eeq
one obtains a preimage, given by
\beq
\label{eqn:canonical_sextic_rank9star}
 \mathcal{S}\colon \quad y^2 = c_0 w^6 +c_2(u, v) w^4 + c_4(u, v) w^2 + c_6(u, v) \, .
\eeq
The choice of 8 general points $s_1,\ldots,s_8$ is essential for the construction of $\mathcal{Z}$, but the choice of the nodal sextic $S'$ defined by $W=0$ is an extra (non-unique) choice. 
\par We make the following:
\begin{remark}
It might be worthwhile pointing out that $\tilde{\mathcal{Z}}$ is the blow up of $\mathbb{P}^2$ in the nine base points of the pencil of cubics spanned by $U,V$ and that this pencil induces the elliptic fibration on $\tilde{\mathcal{Z}}$.
Lemma~\ref{lem:Z} provides an effectiveway to obtain the Weierstrass model of this fibration where the section at infinity is given by the 9th base point (besides $s_1,\ldots,s_8$).\\
\end{remark}
Obviously, there is a double cover $\pi_\mathcal{S}: \mathcal{S}  \rightarrow \mathbb{P}^2 = \mathbb{P}(u, v, w)$. 
We have the following:
\begin{proposition}
\label{prop0}
For general $\mathcal{Z}$, the covering $\pi_\mathcal{S}$ is branched over a smooth irreducible sextic curve $S \subset  \mathbb{P}(u, v, w)$ of genus 10 which admits 120 6-tangent conics.
\end{proposition}
\begin{proof}
The double cover $\pi_\mathcal{S}\colon \mathcal{S}  \rightarrow \mathbb{P}(u, v, w)$ is branched over the curve 
\beq
 S \colon \quad 0 =  c_0 w^6 +c_2(u, v) w^4 + c_4(u, v) \, w^2 + c_6(u, v) \, ,
\eeq
which is generically a smooth plane sextic whence of genus 10. In particular, we must have $c_0 \neq 0$. 
\par It follows from \cite{MR1081832}*{Thm.~10.8} that on a rational elliptic surface the number of sections $P$ with intersection number $P\circ O= 0$ is finite and at most 240. Moreover, every such $P$ has a height pairing $\langle P, P \rangle \le 2$. Since the rational elliptic surface $\widetilde{\mathcal{Z}}$ in Equation~(\ref{eqn:srfc_Z_RES}) is general, there are exactly 240 sections, and their form can be made explicit using \cite{MR1081832}*{Thm.~10.10}. In our situation,  the sections, disjoint from the zero section $O$, must come in pairs $P_{i}^\pm$ of the form $X= p_i(U, V)$, $Y= \pm q_i(U, V)$, $Z=1$ with $i=1, \dots, 120$ where $p_i, q_i$ are non-trivial homogeneous polynomials of degree 2 and 3, respectively. As $P_i^\pm \circ O= 0$, every pair of sections pulls back to a pair of bi-sections $A^\pm_i$ on $\mathcal{S}$ so that
\beq
  w^2 = p_i( u, v) , \quad y = \pm q_i (u, v) 
\eeq
solves Equation~(\ref{eqn:canonical_sextic_rank9star}). Their image under $\pi_\mathcal{S}$ is the conic in $\mathbb{P}(u, v, w)$ given by $C_i \colon 0 = w^2 - p_i( u, v)$. All intersection multiplicities between $C_i$ and $S'$ must be even as $y^2 = q_i (u, v)^2$ is a perfect square for every parametrization of the conic $C_i$. Thus, $C_i$ is a conic tangent to the branch curve in every point of $C_i\cap S$, where $S$ is the branch curve of $\pi_{\mathcal{S}}$.
\end{proof}
\begin{remark}
The degree-one del Pezzo surface $\mathcal{Z}$ contains 240 (-1)curves, and their pull-backs via $\varphi$ on $\mathcal{X}$ are (-2)-curves.  In the proof above, we have seen that these curves come in pairs $A_i^\pm$ with $i=1, \dots, 120$ such that  the image of $A_i^+ + A_i^-$ by $\pi$ is a conic $C_i$ tangent to the branch curve. Using Equation~(\ref{eqn:map_rank9}) the action of $k_\mathcal{S} \colon y \mapsto -y$  corresponds to $R \mapsto -R$ on $\mathcal{Z}$, and $k_\mathcal{X}$ interchanges $A_i^\pm$ and fixes $C_i$. The involution $\jmath_\mathcal{S} \colon w \mapsto -w$ is the covering involution of the double cover $\varphi\colon \mathcal{S} \to \mathcal{Z}$.
\end{remark}
We have the following:
\begin{proposition}
\label{prop1}
The K3 surface $\mathcal{X}$ obtained as the minimal resolution of the general double sextic in Equation~(\ref{eqn:canonical_sextic_rank9star}) has $\mathrm{NS}(\mathcal{X})  \cong H(2) \oplus A_1^{\oplus 7}$. Conversely, every K3 surface $\mathcal{X}$ with $\mathrm{NS}(\mathcal{X})  \cong H(2) \oplus A_1^{\oplus 7}$ is isomorphic to such a K3 surface. 
\end{proposition}
\begin{proof}
We consider the pencil of lines through a node on the sextic $S$. Pulling back the class of the line via the minimal resolution of the double double sextic in Equation~(\ref{eqn:canonical_sextic_rank9star}), this defines an elliptic fibration without section on $\mathcal{X}$. One checks that pull back of the elliptic fiber and the eight classes $ \mathrm{E}_i$ form a lattice isometric to $H(2) \oplus A_1^{\oplus 7}$.  This proves that $\mathrm{NS}(\mathcal{X})  \cong H(2) \oplus A_1^{\oplus 7}$. Conversely, Roulleau proved in \cite{Roulleau22} that every $H(2) \oplus A_1^{\oplus 7}$-polarized K3 surface $\mathcal{X}$ can be obtained as  a double cover of $\mathbb{P}^2$ branched on a plane sextic with eight nodes. In particular, the double sextics of Equation~(\ref{eqn:canonical_sextic_rank9star}) form an 11-dimensional family.
\end{proof}
\par The $(-1)$-curves $ \mathrm{E}_i$ on $\mathcal{Z}$ pull-back via the double cover $\varphi$ to $(-2)$-curves on $\mathcal{X}$, denoted by $A_1, \dots, A_8$. The pull-back of $\mathcal{L}'$ is denoted by $\mathcal{L}$. We set
\beq
\mathpzc{D}_2 = \varphi^* (-K_\mathcal{Z}) = 3 \mathcal{L} - A_1 - \dots - A_8 \, .
\eeq
We have the following:
\begin{proposition}
\label{prop2}
$\mathpzc{D}_2 \in \mathrm{NS}(\mathcal{X})$ is  nef,  of square-two, and invariant under the action of $\imath_\mathcal{X}$. After a suitable choice of coordinates, for the linear system $\vert \mathpzc{D}_2 \vert$ the corresponding rational map is the projection onto $\mathbb{P}(u, v, w)$ in Equation~(\ref{eqn:canonical_sextic_rank9star}).
\end{proposition}
\begin{proof}
One checks invariance by an explicit lattice computation. The rest of the statement was proven in~\cite{Roulleau22}*{Sec.~9.5}.  
\end{proof}
\subsection{\texorpdfstring{$H \oplus A_1^{\oplus 7}$-polarized K3 surfaces}{Families polarized by H + 7 A1}}
We examine double sextics having eight ordinary double points associated with even eights. These double sextics were also considered by Barth \cite{MR1922094}. In order to distinguish these double sextic from those mentioned in Proposition~\ref{prop1}, we refer to the plane sextic curve and the double sextic surface as $S'$ and $\sigma'\colon \mathcal{S}'  \rightarrow \mathbb{P}^2$, respectively. Here, we consider the case where the sextic curve $S'$ splits as the union of a smooth conic and a smooth quartic curve meeting transversally. This is precisely the double cover introduced in Equation~(\ref{eqn:double-sextics_prime_general_intro}) for general polynomials $C, Q$.  We have the following:
\begin{lemma}
\label{lem:ell_fibration_9}
The pencil of lines through a common point of the conic and quartic curve induces a Jacobian elliptic fibration with singular fibers $7 I_2 + 10 I_1$ and trivial Mordell-Weil group.
\end{lemma}
\begin{proof}
Generically, the smooth conic $C$ meets the smooth quartic $Q$ in eight distinct points, $p_i$, $ i = 1, \ldots, 8$. These points result in eight singularities on $\mathcal{S}'$ - all singularities are rational double points of type $A_1$. Denote by $\widetilde{E}_i$ the rational curves on $\mathcal{X}$ needed to resolve these singularities. Let $\widetilde{C}$ and $\widetilde{Q}$ the strict transforms of the conic $C$ and quartic $Q$, respectively. We also denote by $\mathcal{H}$, the hyperplane class corresponding to pre-images on $ \mathcal{X}$ of generic lines in $\mathbb{P}^2$. The Neron-Severi lattice of $\mathcal{X}$ is then generated by classes of: $\mathcal{H}$, $\widetilde{E}_i$, $i=1, \ldots, 8$, $\widetilde{C}$ and $\widetilde{Q}$.     
\par Consider then the pencil of lines in $\mathbb{P}^2$ passing through the point $p_1$. This pencil induces an elliptic fibration on the K3 surface $\mathcal{X}$, corresponding to the base-point free pencil $ \vert \mathcal{H} - \widetilde{E}_1 \vert $. The curve $\widetilde{C}$ is a section in this fibration, while $\widetilde{E}_1$ is a $2$-section and $\widetilde{Q}$ is a $3$-section. Moreover, there are seven singular fibers of Kodaira type $I_2$ and ten singular fibers of Kodaira type $I_1$. The $I_2$ fibers are formed as $\widetilde{E}_i + \widetilde{A}_i $, $i=2, \cdots, 8$, where $\widetilde{A}_i$ is the rational curve obtained as the strict pre-image of the line in $\mathbb{P}^2$ passing through $p_1$ and $p_i$. The section $\widetilde{C}$ meets $\widetilde{E}_i$ but not $\widetilde{A}_i$.  Results in \cites{MR1813537, MR4704757} show that the Mordell-Weil group of this fibration is trivial. 
\end{proof}
For the double sextic $\mathcal{S}'$,  we consider the associated quartic $\mathcal{K} \subset \mathbb{P}^3 = \mathbb{P}(u, v, w, y)$, introduced in Equation~(\ref{eqn:quartic_general_intro}).  In the situation above, an explicit computation yields the following:
\begin{lemma}
\label{lem:sing_of_K_rank9}
$\mathcal{K}$ in Equation~(\ref{eqn:quartic_general_intro})  has a rational double-point at $\mathrm{p}_1\colon \,  [u: v: w : y]= [0: 0: 0: 1]$ of type ${\bf A}_1$.
\end{lemma}
We have the following:
\begin{proposition}
\label{prop:K_rank9}
The K3 surface $\mathcal{X}$ obtained as the minimal resolution of the double sextic $\mathcal{S}'$ in Equation~(\ref{eqn:double-sextics_prime_general_intro}) or, equivalently, the quartic $\mathcal{K}$ in Equation~(\ref{eqn:quartic_general_intro}) for general polynomials $C, Q$ has $\mathrm{NS}(\mathcal{X})  \cong H \oplus A_1^{\oplus 7}$. Conversely, every K3 surface $\mathcal{X}$ with $\mathrm{NS}(\mathcal{X})  \cong H \oplus A_1^{\oplus 7}$ is isomorphic to such a K3 surface.
\end{proposition}
\begin{proof}
Using the same notation as in the proof of Lemma~\ref{lem:ell_fibration_9} we have a lattice isomorphism:
\beqn
  \mathrm{NS}(\mathcal{X}) \simeq \langle \mathcal{H} - \widetilde{E}_1 , \widetilde{C} \rangle  \oplus \langle \widetilde{A}_2 \rangle \oplus \cdots \oplus \langle \widetilde{A}_8 \rangle  \ .
 \eeqn 
This shows that $\mathrm{NS}(\mathcal{X})$ is a lattice of type $H \oplus A_1^{\oplus 7} $.  It remains to show that every K3 surface $\mathcal{X}$ with $\mathrm{NS}(\mathcal{X})  \cong H \oplus A_1^{\oplus 7}$ is isomorphic to the minimal resolution of a double sextic $\mathcal{S}'$.  One has that, after applying an isometry of $\mathrm{NS}(\mathcal{X})$ (see the discussion in \cite{MR2355598}), the hyperbolic factor $H$ will be spanned by the fiber class $F'$ and section $\widetilde{C}$ of a Jacobian elliptic fibration on $\mathcal{X}$ that has seven $I_2$ fibers, the $A_1$-factors will correspond to seven rational curves $\widetilde{A}_i$, $i=2,\ldots, 8$. In this context, the divisor
\beqn 
 \mathcal{H}  =  7 F' + 2 \widetilde{C} - \widetilde{A}_2 - \cdots -\widetilde{ A}_8 
\eeqn 
satisfies $\mathcal{H}^2 =2$, $h^0(\mathcal{H})=3$ and corresponds to a base-point-free linear system that determines map $\mathcal{X} \rightarrow \mathbb{P}^2$. This map is  generically 2-to-1 and branches over a smooth conic and a smooth quartic on $\mathbb{P}^2$.  The conic is the image of the section $\widetilde{C}$. The quartic is the image of a 3-section. Moreover, the above map collapses exactly eight rational curves on $\mathcal{X}$: seven of these are the curves $\widetilde{A}_i$, the eighth curve is a 2-section.   
\end{proof}
The statement that every algebraic K3 surface $\mathcal{X}$ with $\mathrm{NS}(\mathcal{X})  \cong H \oplus A_1^{\oplus 7}$ has a birational model as double sextic associated with an even eight was also given in ~\cite{Roulleau22}*{Prop.~9.4} without proof. 
\section{K3 surfaces of Picard number 10}
\label{sec:rank10}
In Picard number 10, there are two possible lattice polarizations for K3 surfaces $\mathcal{X}$ with $\mathrm{Aut}(\mathcal{X}) \cong (\mathbb{Z}/2\mathbb{Z})^2$, namely $H(2)\oplus D_4^{\oplus 2}\cong H \oplus N$ and  $H \oplus A_1^{\oplus 8}$.
\subsection{The van Geemen-Sarti family}
In Section~\ref{sec:statment} we already introduced the van Geemen-Sarti family. Generically, the alternate fibration in Equation~(\ref{eqn:vgs_intro}) on $\mathcal{X}$ has 8 fibers of type $I_1$ over the zeroes of $a_4^2 - 4 b_8=0$ and 8 fibers of type $I_2$ over the zeroes of $b_8=0$ with Mordell-Weil group $\operatorname{MW}(\mathcal{X}, \pi_\mathcal{X})\cong  \mathbb{Z}/2\mathbb{Z}$. It was shown in \cite{MR2274533}  that for general members of the family we have
\beq
  \mathrm{NS}(\mathcal{X}) \  \cong \ H \oplus N , \qquad
  \mathrm{T}_\mathcal{X} \  \cong \ H^2 \oplus N \, .
\eeq
As previously indicated, the alternate fibration on $\mathcal{X}$ (and therefore the associated Geemen-Sarti involution $\imath_\mathcal{X}$) is unique. 
\par Let $\mathcal{X}$ be a general K3 surface with $\mathrm{NS}(\mathcal{X})  \cong H \oplus N$, equipped with the alternate fibration in Equation~(\ref{eqn:vgs_intro}). Since  $\mathrm{Aut}(\mathcal{X}) \cong (\mathbb{Z}/2)^2$, the van Geemen-Sarti involution $\imath_\mathcal{X}$ and the (fiberwise) hyperelliptic involution $\jmath_\mathcal{X}$ generate $\mathrm{Aut}(\mathcal{X})$. Using the alternate fibration we construct a double sextic $\mathcal{S}$ whose minimal resolution is isomorphic to $\mathcal{X}$. However, the equation for such a double sextic is not canonical:
\begin{proposition}
\label{prop:rank10}
For a factorization $b_8(u, v) = c_3(u, v) \cdot d_5(u, v)$ where $c_3$ and $d_5$ are homogeneous polynomials of degree 3 and 5, respectively, $\mathcal{X}$ is isomorphic to the minimal resolution of the double sextic given by
\beq
\label{eqn:double_sextic_rank10}
  \mathcal{S}\colon \quad y^2   = w \big( c_3(u, v)  \, w^2 + a_4(u, v) \, w + d_5(u, v)  \big) \,.
\eeq
In the general case, the branch curve  is the union of the line $w=0$ and an irreducible quintic, and the double sextic $\mathcal{S}$ has singularities of type ${\bf D}_4$ at $[u:v:w:y]=[0:0:1:0]$ and ${\bf A}_1$ at $[u_0:v_0:0:0]$ with $d_5(u_0, v_0)=0$.
\end{proposition}
\begin{proof}
The transformation
\beq
\big( u , v , w, y \big) = \big( c_3(U, V) U Z, c_3(U, V) V Z, X, c_3(U, V)^2 Y Z^2 \big)
\eeq
has the property that the right side transforms under rescaling
\beq
 (U, V, X, Y, Z) \mapsto (\lambda U, \lambda V, X, \lambda^{-2} Y,  \lambda^{-4} Z), \qquad
 (U, V, X, Y, Z) \mapsto (U, V, \lambda X, \lambda Y,  \lambda Z) \,,
\eeq 
for $\lambda \in \mathbb{C}^\times$, with weights $(0, 0, 0, 0)$ and $(1, 1, 1, 3)$, respectively. Thus, it induces a rational map from $\mathcal{X}$ in Equation~(\ref{eqn:vgs_intro}), considered as a double cover of $\mathbb{F}_4$, to $\mathcal{S}$ in Equation~(\ref{eqn:double_sextic_rank10}). For any $Z_0 \in \mathbb{C}^\times$ the transformation
\beq
 \big([U:V], [X: Y : Z] \big) = \big( [u:v], [c_3(u, v) w Z_0 : c_3(u, v) y Z_0: Z_0  ]\big)
\eeq
provides a rational inverse.
\par Next, we prove that $\mathcal{S}$ in Equation~(\ref{eqn:double_sextic_rank10}) is well defined. For any other factorization $b_8(u, v) = \tilde{c}_3(u, v) \cdot \tilde{d}_5(u, v)$ the birational coordinate change  $[\tilde{y}: \tilde{w}: \tilde{u}: \tilde{v}] = [  \tilde{c}_3(u, v) c_3(u, v)^2 y:  \ \tilde{c}_3(u, v) w : \ c_3(u, v) u: \ c_3(u, v) v]$ transforms the equation
\beq
\label{eqn:S_H+N}
  \tilde{y}^2   = \tilde{w} \big( \tilde{c}_3(\tilde{u}, \tilde{v})  \, \tilde{w}^2 + a_4(\tilde{u}, \tilde{v}) \, \tilde{w} + \tilde{d}_5(\tilde{u}, \tilde{v})  \big) 
\eeq
into Equation~(\ref{eqn:double_sextic_rank10}). The rest of the statement is immediate.
\end{proof}
\begin{figure}
  	\centering
	\scalebox{0.7}{%
    		\begin{tikzpicture}[rotate=90]
       		\draw (0,1.5) -- (3.5,0.5);
\draw (0,1.5) -- (2.5,0.5);
\draw (0,1.5) -- (1.5,0.5);
\draw (0,1.5) -- (0.5,0.5);
\draw (0,1.5) -- (-0.5,0.5);
\draw (0,1.5) -- (-1.5,0.5);
\draw (0,1.5) -- (-2.5,0.5);
\draw (0,1.5) -- (-3.5,0.5);

\draw [very thick] (3.5,0.5) -- (3.5,-0.5);
\draw [very thick] (2.5,0.5) -- (2.5,-0.5);
\draw [very thick] (1.5,0.5) -- (1.5,-0.5);
\draw [very thick] (0.5,0.5) -- (0.5,-0.5);
\draw [very thick] (-3.5,0.5) -- (-3.5,-0.5);
\draw [very thick] (-2.5,0.5) -- (-2.5,-0.5);
\draw [very thick] (-1.5,0.5) -- (-1.5,-0.5);
\draw [very thick] (-0.5,0.5) -- (-0.5,-0.5);

\draw (0,-1.5) -- (3.5,-0.5);
\draw (0,-1.5) -- (2.5,-0.5);
\draw (0,-1.5) -- (1.5,-0.5);
\draw (0,-1.5) -- (0.5,-0.5);
\draw (0,-1.5) -- (-0.5,-0.5);
\draw (0,-1.5) -- (-1.5,-0.5);
\draw (0,-1.5) -- (-2.5,-0.5);
\draw (0,-1.5) -- (-3.5,-0.5);

\draw (0,1.5) node {$\bullet$};
\draw (3.5,0.5) node  {$\bullet$};
\draw (2.5,0.5) node  {$\bullet$};
\draw (1.5,0.5) node  {$\bullet$};
\draw (0.5,0.5) node  {$\bullet$};
\draw (-0.5,0.5) node  {$\bullet$};
\draw (-1.5,0.5) node  {$\bullet$};
\draw (-2.5,0.5) node  {$\bullet$};
\draw (-3.5,0.5) node  {$\bullet$};

\draw (-0.3,1.8) node [above]{$A_{1}$};
\draw (3.2,0.3) node [left]{$A_{9}$};
\draw (2.2,0.2) node [left]{$A_{8}$};
\draw (1.2,0.2) node [left]{$A_{7}$};
\draw (0.2,0.2) node [left]{$A_{6}$};
\draw (-0.8,0.2) node [left]{$A_{5}$};
\draw (-1.8,0.2) node [left]{$A_{4}$};
\draw (-2.8,0.2) node [left]{$A_{3}$};
\draw (-3.8,0.2) node [left]{$A_{2}$};

\draw (0,-1.5) node {$\bullet$};
\draw (3.5,-0.5) node  {$\bullet$};
\draw (2.5,-0.5) node  {$\bullet$};
\draw (1.5,-0.5) node  {$\bullet$};
\draw (0.5,-0.5) node  {$\bullet$};
\draw (-0.5,-0.5) node  {$\bullet$};
\draw (-1.5,-0.5) node  {$\bullet$};
\draw (-2.5,-0.5) node  {$\bullet$};
\draw (-3.5,-0.5) node  {$\bullet$};

\draw (0.3,-1.9) node [below]{$A_{18}$};
\draw (3.2,-0.7) node [left]{$A_{17}$};
\draw (2.2,-0.7) node [left]{$A_{16}$};
\draw (1.2,-0.7) node [left]{$A_{15}$};
\draw (0.2,-0.7) node [left]{$A_{14}$};
\draw (-0.8,-0.7) node [left]{$A_{13}$};
\draw (-1.8,-0.7) node [left]{$A_{12}$};
\draw (-2.8,-0.7) node [left]{$A_{11}$};
\draw (-3.8,-0.8) node [left]{$A_{10}$};

    		\end{tikzpicture}}
\caption{Dual graph for $\mathrm{NS}(\mathcal{X}) = H \oplus N$}
\label{fig:pic10}
\end{figure}
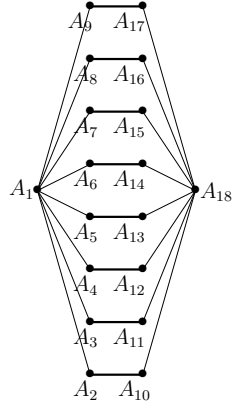
The lattice $L=H \oplus N$ and the configuration of its finite set of $(-2)$-curves was analyzed in detail in~\cite{Roulleau22}*{Sec.~10.6}.  We will use the same notation $A_1, \dots, A_{18}$ to denote the (-2)-curves in the dual graph of $\mathcal{X}$ as shown in Figure~\ref{fig:pic10}. Note that for simplicity the graph does not show \emph{all} rational curves on a general K3 surface with $\mathrm{NS}(\mathcal{X})  \cong H \oplus N$.  We have the following:
\begin{proposition}
\label{prop:symmetry_rank10a}
$\mathpzc{D}_2 = 2A_{1}+2A_{2}+A_{3}+A_{4}+A_{10} \in \mathrm{NS}(\mathcal{X})$ is  nef,   base-point free, of square-two, and invariant under the action of $\imath_\mathcal{X}$. After a suitable choice of coordinates, for the linear system $\vert \mathpzc{D}_2 \vert$ the corresponding rational map is the projection onto $\mathbb{P}(u, v, w)$ in Equation~(\ref{eqn:S_H+N}).
\end{proposition}
\begin{proof}
One checks invariance by an explicit lattice computation. The rest of the statement was shown in~\cite{Roulleau22}*{Sec.~10.6}.
\end{proof}
\par A general K3 surface $\mathcal{X}$ with $\mathrm{NS}(\mathcal{X})  \cong H \oplus N$ can also be obtained as the minimal resolution of a double cover of $\mathbb{F}_0 = \mathbb{P}(u, v) \times \mathbb{P}(\tilde{X}, \tilde{Z})$. Given a factorization $b_8(u, v) = \tilde{b}_4(u, v) \cdot \tilde{b}'_4(u, v)$ into homogeneous polynomials of degree four, $\mathcal{X}$ is obtained as the minimal resolution of a double quadric surface, defined by the equation
\beq
\label{eqn:vg_tilde}
  \widetilde{\mathcal{X}}\colon \quad \tilde{Y}^2  = \tilde{X}\tilde{Z} \big( \tilde{b}_4(u, v) \tilde{X}^2 + a_4(u, v)  \tilde{X} \tilde{Z} + \tilde{b}'_4(u, v)  \tilde{Z}^2  \big) \,.
\eeq
The surface is branched on a divisor class of bidegree $(4,4)$. The class decomposes into the union of divisors of bidegree $(0,1) + (0,1) + (4,2)$. These divisors are two $(-2)$-curves and the class of a curve of bidegree $(4,2)$. The former is of geometric genus 3 and given by the equation $\tilde{b}_4(u, v)  \tilde{X}^2 + a_4(u, v)  \tilde{X} \tilde{Z} + \tilde{b}'_4(u, v)  \tilde{Z}^2=0$. The latter two are the classes of the section of the alternate fibration, given by $\tilde{Z}=0$, and the 2-torsion section, given by $\tilde{X}=0$. The divisors form the invariant locus of the antisymplectic involution $\jmath_{\widetilde{\mathcal{X}}}\colon \tilde{Y} \mapsto -\tilde{Y}$, which is consistent with Nikulin's theorem \cite{MR633160b}*{Thm.~ 4.2.2}; see Remark~\ref{rem:branch_locus}.
\subsection{\texorpdfstring{$H \oplus A_1^{\oplus 8}$-polarized K3 surfaces}{Families polarized by H + 8 A1}}
The lattice $L=H \oplus A_1^{\oplus 8}$ and the configuration of its finite set of $(-2)$-curves was analyzed in detail in~\cite{Roulleau22}*{Sec.~10.8}.  We will use the notation $\mathrm{f}_1, \mathrm{f}_2, \mathrm{e}_1, \dots, \mathrm{e}_8$ to denote the canonical basis of $L$.  
\par We start by considering the double sextic $\mathcal{S}'$ in Equation~(\ref{eqn:double-sextics_prime_general_intro}),  assuming that the quartic $Q$ has a single node. We have the following analogue of Lemma~\ref{lem:ell_fibration_9}:
\begin{lemma}
\label{lem:ell_fibration}
If the sextic curve $S'$ splits as the union of a smooth conic and a uninodal quartic curve meeting transversally, then the pencil of lines through a common point of the conic and quartic curve induces a Jacobian elliptic fibration with singular fibers $8 I_2 + 8 I_1$ and trivial Mordell-Weil group. 
\end{lemma}
\begin{proof}
The proof is analogous to the Proof of Lemma~\ref{lem:ell_fibration_9}: the conic and the 1-nodal quartic have 8 distinct intersection points. It is then clear that the inverse image in the K3 of the pencil of lines through one intersection point gives an elliptic fibration, with the  section given by the preimage of the conic. Moreover, there are $I_2$-fibers over the seven lines that contain another point of intersection as well as over the line that passes through the node. 
\end{proof}
\begin{remark}
\label{rem:embedding}
On a K3 surface $\mathcal{X}$, a Jacobian elliptic fibration with $8 I_2 + 8 I_1$ and trivial Mordell-Weil group identifies the canonical basis of $L$, denoted by $\mathrm{f}_1$, $\mathrm{f}_2$, $\mathrm{e}_1, \dots, \mathrm{e}_8$, with classes in $\mathrm{NS}(\mathcal{X})$. We set
\beq
 F' = \mathrm{f}_1 , \quad  A_0 = -\mathrm{f}_1 + \mathrm{f}_2, \quad  A_i = \mathrm{f}_1 - \mathrm{e}_i \; \text{for $i=1, \dots, 8$,}
\eeq
and let $F'$ be a fiber class,  $A_0$ the class of a section with $F' \cdot A_0=1$, and $A_i + (F' - A_i)$ the reducible fibers 
for $i=1, \dots, 8$.  The section intersects the reducible fibers according to $A_0 \cdot A_i=1$ and $A_0 \cdot (F' - A_i)=0$ for $i=1, \dots, 8$. 
$\mathrm{NS}(\mathcal{X})$ is generated by $\langle F' + A_0, F', F' - A_1, \dots, F' - A_8\rangle$.
\end{remark}
\par Analogous to the Proposition~\ref{prop:K_rank9}, every algebraic K3 surface $\mathcal{X}$ with $\mathrm{NS}(\mathcal{X})  \cong H \oplus A_1^{\oplus 8}$ has a birational model as double sextic $\mathcal{S}'$ in Equation~(\ref{eqn:double-sextics_prime_general_intro}) where $S'$ splits as the union of a smooth conic and a uninodal quartic curve meeting transversally.  We have the following:

\begin{proposition} 
\label{prop:double_surface_Sp_10}
Let $\mathcal{X}$ be a general K3 surface with $\mathrm{NS}(\mathcal{X})  \cong H \oplus A_1^{\oplus 8}$. Then $\mathcal{X}$ is isomorphic to the minimal resolution of the double sextic $\mathcal{S}'$ such that the branch curve $S'$ splits as the union of a smooth conic and a uninodal quartic curve meeting transversally.
\end{proposition}
\begin{proof}
The arguments are similar to the ones in the proof of Proposition~\ref{prop:K_rank9}. The node in the quartic curve accounts for the additional $A_1$ factor in the Neron-Severi lattice $\mathrm{NS}(\mathcal{X})$.  
\end{proof}
\par In the situation of Proposition~\ref{prop:double_surface_Sp_10}, we move the node to $\mathrm{n} = [0: 0: 1]$ in $\mathbb{P}(u, v, w)$. Since the conic is assumed to be smooth, we can bring it into the form $h_0 w^2 + k_1(u, v) w + j_2(u, v)$  by coordinate shifts which keep the position of $\mathrm{n}$ fixed.  One obtains for the conic $C = h_0 w^2+ v w +  j_0  u^2$. Moreover, the conic contains $\mathrm{n}$ if and only if $h_0 = 0$.  For $h_0 j_0 \not = 0$ the conic can further be brought into the form $0 = w^2 - u v$  while keeping the position of the node $\mathrm{n}$ fixed.  Thus, we are led to consider the double sextic $\mathcal{S}'$ in Equation~(\ref{eqn:double-sextics_prime_general_intro}) with
\beq
\label{eqn:double_sextic10_polys}
  C = w^2 - uv  \, , \qquad  Q= c_2(u, v) \, w^2 + e_3(u, v) \, w + d_4(u, v) \, ,
\eeq
where $c_2, e_3, d_4$ are homogeneous polynomials of degree 2, 3, and 4, respectively. In this way, we obtain  the double sextic $\mathcal{S}'$, given by
\beq
\label{eqn:canonical_sextic_prime_rank10star}
  \mathcal{S}'\colon \quad \tilde{y}^2   = \big(w^2 - uv\big) \big( c_2(u, v) \, w^2 + e_3(u, v) \, w + d_4(u, v) \big) \, ,
\eeq
where the branch curve $S'$ splits as the union of a smooth conic and a uninodal quartic curve. We also consider the associated quartic $\mathcal{K}$ in Equation~(\ref{eqn:quartic_general_intro}) using the polynomials in~(\ref{eqn:double_sextic10_polys}), this is, the quartic given by 
\beq
\label{eqn:canonical_quartic_rank10star}
  \mathcal{K}\colon \quad \big(w^2 - uv\big)  y ^2 = c_2(u, v) \, w^2 + e_3(u, v) \, w + d_4(u, v) \,.
\eeq
One has the following:
\begin{lemma}
\label{lem:sings_K_rank10}
$\mathcal{K}$ has two rational double-point singularities of type $\mathbf{A}_1$ at
\beq
\label{eqn:singular_points_10}
   \mathrm{p}_1\colon \  [u: v: w : y]= [0: 0: 0: 1] \,, \qquad   \mathrm{p}_2\colon \  [u: v: w : y]= [0 : 0: 1: 0] \,.
\eeq  
\end{lemma}
We also consider the double sextic $\mathcal{S}$ given by
\beq
\label{eqn:canonical_sextic_rank10star}
  \mathcal{S}\colon \quad y^2   = uv w^4 -  a_4(u, v)  w^2 +  b''_6(u, v) \,,
\eeq
where $a_4  =    u v c_2 - d_4$ and $b''_6 = e_3^2/4 - c_2 d_4$. One easily checks the following:
\begin{lemma}
\label{lem:sings_on_S_rank10}
$\mathcal{S}$  has a  rational double-point singularity of type $\mathbf{A}_1$ at  
\beqn
 \mathrm{p}\colon \  [u: v: w : y]= [0: 0: 1:0] \in \mathbb{P}(1,1,1,3).
 \eeqn
\end{lemma}
In the situation above we have the following:
\begin{lemma}
\label{lem:birational}
The minimal resolutions of $\mathcal{S}'$, $\mathcal{S}$, $\mathcal{K}$ are isomorphic. 
\end{lemma}
\begin{proof}
To avoid  confusion between different coordinate sets, we write the double sextic $\mathcal{S}$ as follows:
\beq
\label{eqn:canonical_sextic_rank10star_pf}
  \mathcal{S}\colon \quad \eta^2   = uv  \xi^4 -  a_4(u, v)  \xi^2 +  b''_6(u, v)\,.
\eeq
Birational transformations between the surfaces $\mathcal{S}'$, $\mathcal{S}$, $\mathcal{K}$ are given by
\beq
 \xi = \frac{\tilde{y}}{w^2 - uv} =y\, , \quad \eta =y^2 w - c_2 w - \frac{e_3}{2}\, ,
\eeq
and
\beq
\begin{split}
 w  = - \frac{2(uv c_2 + d_4) \eta + 2 uv e_3 \xi^2 - (uv c_2 - d_4) e_3}{2 e_3 \eta - 2(uv c_2 + d_4) \xi^2 + 2 c_2 (uv c_2 + d_4) -e_3^2}\, , \\
 \tilde{y}  =  \frac{2(- (uv c_2 + d_4)^2 + uv e_3^2) \xi }{2 e_3 \eta - 2(uv c_2 + d_4) \xi^2 + 2 c_2 (uv c_2 + d_4) -e_3^2}\,, \quad  y=\xi \,,
\end{split}
\eeq
respectively.
\end{proof}
From Proposition~\ref{prop:double_surface_Sp_10} and Lemma~\ref{lem:birational} we have the immediate:
\begin{corollary}
\label{cor:rank10b}
The K3 surface $\mathcal{X}$ obtained as the minimal resolution of the double sextic $\mathcal{S}'$ in Equation~(\ref{eqn:canonical_sextic_prime_rank10star}) or $\mathcal{S}$ in~(\ref{eqn:canonical_sextic_rank10star}) has $\mathrm{NS}(\mathcal{X})  \cong H \oplus A_1^{\oplus 8}$. Conversely, every K3 surface $\mathcal{X}$ with $\mathrm{NS}(\mathcal{X})  \cong H \oplus A_1^{\oplus 8}$ is isomorphic to such a K3 surface.
\end{corollary}
\par In the following proposition we use the identification between lattice vectors $\mathrm{f}_i, \mathrm{e}_j$ and classes in $\mathrm{NS}(\mathcal{X})$ that was established in Remark~\ref{rem:embedding}.  We have the following:
\begin{proposition}
\label{prop:symmetry_rank10b}
$\mathpzc{D}_2 = 3 \mathrm{f}_1 + 3 \mathrm{f}_2 - \sum_{i=1}^8 \mathrm{e}_i \in \mathrm{NS}(\mathcal{X})$ is  nef,  base-point free, of square-two, and invariant under the action of $\imath_\mathcal{X}$. After a suitable choice of coordinates, for the linear system $\vert \mathpzc{D}_2 \vert$ the corresponding rational map is the projection onto $\mathbb{P}(u, v, w)$ in Equation~(\ref{eqn:canonical_sextic_rank10star}).
\end{proposition}
\begin{proof}
One checks invariance by an explicit lattice computation. The rest of the statement was shown in~\cite{Roulleau22}*{Sec.~10.8.1} 
\end{proof}
By $\mathrm{a}_1$ and $\mathrm{b}_1$ we denote the exceptional curves obtained in the minimal resolution of the singularity at $\mathrm{p}_1$ and $\mathrm{p}_2$ on $\mathcal{K}$, respectively, in Lemma~\ref{lem:sings_K_rank10}. We state the main result for this section:
\begin{proposition}
\label{prop_divisor_rank10}
The polarizing divisor of $\mathcal{K}$ is $\mathpzc{H} = \mathpzc{D}_2 + \mathrm{b}_1$  where $\mathrm{b}_1$ is the class of an $\mathrm{e}_i$ for some $i=1, .., 8$. In particular, one can set $\mathrm{b}_1=\mathrm{e}_8$. Moreover, one has $\mathrm{a}_1 = - \mathrm{f}_1 + \mathrm{f}_2$.
\end{proposition}
\begin{proof}
The proof follows from the fact that projection from a node gives the relation $\mathpzc{H} = \mathpzc{D}_2 + n$ where $n$ is the class of the $(-2)$ curve over the node. It remains to provide the geometric interpretation for the classes $n$ and  $A_0 = - \mathrm{f}_1 + \mathrm{f}_2$ in terms of the quartic $\mathcal{K}$. As explained before, $A_0=- \mathrm{f}_1 + \mathrm{f}_2$ is the class of the section for the fibration $F'$. Moreover, projecting $\mathcal{K}$ with center $\mathrm{p}_1$ onto $\mathbb{P}(u, v, w)$ yields the double sextic $\mathcal{S}'$, and one checks that the exceptional curve obtained in the minimal resolution of the singularity at $\mathrm{p}_1$ is $A_0$ as it is the preimage of $C$. On the other hand, projecting $\mathcal{K}$ from $\mathrm{p}_2$ yields the double sextic $\mathcal{S}$. Using the description in the proof of Lemma~\ref{lem:ell_fibration},  one finds that the exceptional curve obtained in the minimal resolution of the singularity at $\mathrm{p}_2$ is the non-neutral component of one reducible fiber of type $A_1$ in the elliptic fibration with fiber $F'$, that is, $F' - A_i = \mathrm{e}_i$ for some $i \in \{1, \dots, 8\}$. 
\end{proof}
\section{K3 surfaces of Picard number greater than 10}
\label{sec:rank11ff}
Here, we will construct birational models as double sextics and quartics for the K3 surfaces $\mathcal{X}$ with $\mathrm{NS}(\mathcal{X})  \cong L$, for $L$ in Equation~(\ref{eqn:lattices}) with $\rho_L > 10$. Their explicit normal forms will be derived from the alternate fibration.  
\subsection{Uniqueness of the alternate fibration}
Let us explain the meaning of uniqueness for the alternate fibration in Equation~(\ref{eqn:vgs_intro}).  Given a Jacobian elliptic fibration on $\mathcal{X}$, the classes of fiber and section span a rank-two primitive sub-lattice of $\operatorname{NS}(\mathcal{X})$ isomorphic to $H$. The converse also holds: given a primitive lattice embedding $H \hookrightarrow \operatorname{NS}(\mathcal{X})$ whose image contains a pseudo-ample class, it is known from \cite{MR2355598}*{Thm.~2.3} that there exists a Jacobian elliptic fibration on the surface $\mathcal{X}$, whose fiber and section classes span $H$. For a primitive lattice embedding $j \colon H \hookrightarrow L$  we denote by $K= j(H)^{\perp}$ the orthogonal complement in $L$ and by $K^{\text{root}}$ the sub-lattice spanned by the roots of $K$, i.e., the algebraic classes of self-intersection $-2$  in $K$. We also introduce the factor group $\mathpzc{W}= K/ K^{\text{root}}$.  The pair $(K^{\text{root}}(-1), \mathpzc{W})$ is called the \emph{frame} associated with the Jacobian elliptic fibration.  A frame determines the root lattice $K^{\text{root}}$ attached to the reducible fibers and the Mordell-Weil group $\mathpzc{W}$ of a Jacobian elliptic fibration uniquely.
\par One can also ask for a more precise classification of Jacobian elliptic fibrations, up to automorphisms of the surface $\mathcal{X}$. The difference between these classifications is explained in \cite{Braun:2013aa}. In the situation above, assume that we have a second primitive embedding $j' \colon H \hookrightarrow L$, such that the orthogonal complement of the image $j'(H)$, is isomorphic to the given lattice $K$. One would like to see whether $j$ and $j'$ correspond to Jacobian elliptic fibrations isomorphic under $\mathrm{Aut}(\mathcal{X})$ or not. 
\par By standard arguments (see  \cite{MR525944}*{Prop.~ 1.15.1}), in the situation above there will exist an isometry $\gamma \in O(L) $ such that $j' = \gamma \circ j$. An induced element $\gamma^* \in O(D(K)) $ can then be obtained as the image of $\gamma$ under the group homomorphism $O(L)  \rightarrow O(D(L))  \cong  O(D(K))$ where the isomorphism is due to the decomposition $L = j(H) \oplus K$ and, as such, it depends on the lattice embedding $j$. Denote the group $O(D(K)) $ by $\mathpzc{A}$ and consider the following two subgroups of $\mathpzc{A}$: the first subgroup $\mathpzc{B} \leqslant \mathpzc{A}$ is the image of the following group homomorphism:
\beq
\label{eqn:map_B}
O(K) \cong
\big \{ 
\varphi \in O(L) \ \vert \  \varphi \circ j(H) = j(H) 
\big \} 
 \ \hookrightarrow \ O(L) \ \rightarrow \ O\big(D(L)\big)  \cong O\big(D(K)\big) \, .
\eeq
The second subgroup $\mathpzc{C} \leqslant \mathpzc{A}$ is the image of following group homomorphism:
\beq
O_h(\mathrm{T}_{\mathcal{X}}) \ \hookrightarrow \  O(\mathrm{T}_{\mathcal{X}}) 
\ \rightarrow \ O \big( D ( \mathrm{T}_{\mathcal{X}} ) \big) \cong  O\big(D(L)\big)  \cong O\big(D(K)\big) \, .
\eeq
Here $\mathrm{T}_{\mathcal{X}} $ denotes the transcendental lattice of $\mathcal{X}$ and $O_h(\mathrm{T}_{\mathcal{X}})$ is given by the isometries of $\mathrm{T}_{\mathcal{X}} $ that preserve the Hodge decomposition. Note that one has $D(\operatorname{NS}(\mathcal{X})) \simeq D(\mathrm{T}_{\mathcal{X}}) $  with $q_L = -q_{\mathrm{T}_{\mathcal{X}}} $, as $\operatorname{NS}(\mathcal{X})=L$ and $\mathrm{T}_{\mathcal{X}}$ is the orthogonal complement of $\operatorname{NS}(\mathcal{X}) $ with respect to an unimodular lattice. 
\par It was proved in  \cite{FestiVeniani20}*{Thm~2.8} that the map
\beq
\label{corresp33} 
 H \  \stackrel{j}{\hookrightarrow} \ L \qquad  \rightsquigarrow \qquad \mathpzc{C} \,  \gamma^* \mathpzc{B} \,,
 \eeq
establishes a one-to-one correspondence between Jacobian elliptic fibrations on $\mathcal{X}$ with $ j(H)^{\perp} \simeq K$, up to the action of the automorphism group ${\rm Aut}(\mathcal{X})$, and elements of a double coset. The number of elements of  $\mathpzc{C}  \backslash \mathpzc{A}/ \mathpzc{B}$ is referred by Festi and Veniani as the \emph{multiplicity} of the frame. 
\par The frame of the alternate fibration~(\ref{eqn:vgs_intro}) is $(K^{\text{root}}, \mathpzc{W}) =(A_1^{\oplus 8}, \mathbb{Z}/2\mathbb{Z})$, and van Geemen and Sarti proved that its multiplicity equals one. Two of the authors proved the following extension \cite{Clingher:2022}*{Theorem 2.3}:
\begin{proposition}
\label{prop:frame}
For every lattice $L$ in Equation~(\ref{eqn:lattices}) with $\rho_L>10$ or $L = H \oplus N$ there is a unique frame with $\mathpzc{W}=\mathbb{Z}/2\mathbb{Z}$ and its multiplicity equals one.
\end{proposition}
In Table~\ref{tab:extension0} we present a comprehensive list of all 2-elementary lattices $L$ such that the general $L$-polarized K3 surface $\mathcal{X}$ satisfies $\mathrm{Aut}(\mathcal{X})=(\mathbb{Z}/2\mathbb{Z})^2$. The lattice $L=H \oplus A_1^{\oplus 8}$ cannot be framed with $\mathpzc{W}=\mathbb{Z}/2\mathbb{Z}$, and the general $L$-polarized K3 surface lacks an alternative fibration. Nonetheless, there is a Jacobian elliptic fibration with trivial Mordell-Weil group, but the frame's multiplicity is not one; the multiplicity was computed in \cite{MR4704757}. For $\rho_L=9$, the lattice cannot be decomposed as $H \oplus K$, and the general $L$-polarized K3 surface does not admit a Jacobian elliptic fibration. 
\par As we shall see, it is precisely the uniqueness of the alternate fibration~(\ref{eqn:vgs_intro}) that will allow us to construct a canonical birational model for every general $L$-polarized K3 surface with $\rho_L>10$.
\begin{table}[!ht]
\begin{tabular}{l|c||l|lc}
$\rho_L$ 	& $(\ell_L, \delta_L)$ 	&  \multicolumn{1}{c|}{$L$}	&  \multicolumn{1}{c}{$K^{\text{root}}$} 	& $\mathpzc{W}$ \\
\hline
\hline
$9$		& $(9,1)$				& $H(2) \oplus  A_1^{\oplus 7}$				& --				& --\\
\hline
$10$		& $(8,1)$				& $H \oplus A_1^{\oplus 8}$					& $8 A_1$ 		& $\lbrace \mathbb{I} \rbrace$ \\
\hdashline
$10$		& $(6,0)$				& $H(2)\oplus D_4^{\oplus 2}\cong H \oplus N$	& $8 A_1$ 		& $\mathbb{Z}/2\mathbb{Z}$ \\
\hline
$11$		& $(7,1)$				& $H\oplus D_4 \oplus A_1^{\oplus 5}$ 		& $9 A_1$ 		& $\mathbb{Z}/2\mathbb{Z}$ \\
\hline
$12$		& $(6,1)$				&  $H\oplus D_6 \oplus A_1^{\oplus 4}$ 	 	& $D_4 + 6 A_1$ 	& $\mathbb{Z}/2\mathbb{Z}$ \\
		&					&  $\quad \cong H \oplus D_4^{\oplus 2}  \oplus A_1^{\oplus 2}$ &&\\
\hline
$13$		& $(5,1)$				& $H\oplus E_7 \oplus A_1^{\oplus 4}$ 		& $D_6 + 5 A_1$ 	& $\mathbb{Z}/2\mathbb{Z}$\\
		&					& $\quad \cong H \oplus D_8  \oplus A_1^{\oplus 3}$  &&\\
		&					& $\quad \cong H \oplus D_6  \oplus D_4 \oplus A_1$ &&\\
\hline
$14$		& $(4,0)$				& $H\oplus D_8 \oplus D_4$ 				&  $E_7 + 5 A_1$ 	& $\mathbb{Z}/2\mathbb{Z}$ \\
\hdashline
$14$		& $(4,1)$				& $H\oplus E_8 \oplus A_1^{\oplus 4}$ 		&  $D_8 + 4 A_1$ 	& $\mathbb{Z}/2\mathbb{Z}$\\
		&					& $\quad \cong H \oplus D_{10} \oplus A_1^{\oplus 2}$ &&\\
		&					& $\quad \cong H \oplus E_7 \oplus D_4 \oplus A_1$	 &&\\
		&					& $\quad \cong H \oplus D_{6}^{\oplus 2}$ &&\\
\hline		
$15$		& $(3,1)$				& $H \oplus D_{12} \oplus A_1$ 			& $D_{10} + 3 A_1$	&$\mathbb{Z}/2\mathbb{Z}$\\
		&					& $\quad \cong  H \oplus E_8  \oplus D_4 \oplus A_1$ &&\\
		&					& $\quad \cong H \oplus E_7 \oplus D_6$ &&\\
\hline
$16$		& $(2,1)$				& $H \oplus D_{14}$ 						& $D_{12} + 2 A_1$ 	&$\mathbb{Z}/2\mathbb{Z}$\\
		&					& $\quad \cong H \oplus E_8  \oplus D_6$ &&\\
		&					& $\quad \cong H \oplus E_7 \oplus E_7$  &&\\
\hline
$17$		& $(1,1)$				& $H \oplus E_8  \oplus E_7$ 			& $D_{14} + A_1$ 	& $\mathbb{Z}/2\mathbb{Z}$\\	
\hline
$18$		& $(0,0)$				& $H \oplus E_8  \oplus E_8$ 			& $D_{16}$ 		& $\mathbb{Z}/2\mathbb{Z}$\\	
\hline
\end{tabular}
\captionsetup{justification=centering}
\caption{All 2-elementary lattices with $\mathrm{Aut}(\mathcal{X})=(\mathbb{Z}/2\mathbb{Z})^2$} and a canonical frame
\label{tab:extension0}
\end{table}
\subsection{\texorpdfstring{$L$-polarized K3 surfaces}{Families polarized by L}}
Let $\mathcal{X}$ be a K3 surface with $\mathrm{NS}(\mathcal{X})  \cong L$, for $L$ in Equation~(\ref{eqn:lattices}) and $\rho_L = 10+n$, $1 \le n \le 8$. It follows that $\mathcal{X}$ admits an alternate fibration, given by Equation~(\ref{eqn:vgs_intro}). The fibration was constructed for each lattice in \cites{Clingher:2022, Clingher:2020baq, MR4544843}.
\par It is always possible to move one reducible fiber in the frame determined by Proposition~\ref{prop:frame} to $v=0$, and in doing so, the polynomial $b_8$ in the alternate fibration~(\ref{eqn:vgs_intro}) takes the form
\begin{equation}
\label{eqn::poly_b}
b_8(u, v) = \frac{1}{4}a_4(u, v)^2 - v^2 b''_6(u, v) \, ,
\end{equation}
where $b''_6$ is a homogeneous polynomial of degree 6.
\par This can be seen as follows:  from Table~\ref{tab:extension0} one sees that the reducible fibers of the alternate fibration are $D_{2n} + (8-n) A_1$ for $2 \le n \le 8$ or $E_7 + 5 A_1$ (as a secondary case for $n=4$) and $9 A_1$ for $n=1$. For Picard number 11, the alternate fibration has the following property: for 8 reducible fibers of type $A_1$, the section and the 2-torsion section intersect different components of each reducible fiber, and for an additional reducible fiber of type $A_1$ they intersect the same component. We move the base point of this last fiber to $[u:v]=[1:0]$.  For higher Picard number, the alternate fibration has a fiber of type $D_{2n}$ (for $n=2, \dots, 8$) or a fiber of type $E_7$ (as a second case for $n=4$) which can be moved to $v=0$.  It follows easily that, in all cases, we have Equation~(\ref{eqn::poly_b}) where $b''_6$ is a polynomial of degree 6, as claimed.
\par Using the unique alternate fibration and polynomials $a_4, b_8, b''_6$, using the normalization of Equation~(\ref{eqn::poly_b}), we can construct the K3 surface $\mathcal{X}$ as the minimal resolution of a canonical double sextic:
\begin{proposition}
\label{prop:double_surface_S_rank11ff}
Let $\mathcal{X}$ be a general K3 surface with $\mathrm{NS}(\mathcal{X})  \cong L$, for $L$ in Equation~(\ref{eqn:lattices}) with $\rho_L =10+n$, $1 \le n \le 8$. Then $\mathcal{X}$ is isomorphic to the minimal resolution of the (canonical) double sextic 
\beq
\label{eqn:double-sextics_general_rang_ge11}
   \mathcal{S}\colon \quad  y^2 =   v^2 w^4 - a_4(u, v) w^2  + b''_6(u, v) \, ,
\eeq
where $a_4, b''_6$ are the polynomials using the normalization of Equation~(\ref{eqn::poly_b}). In particular, the double sextic has  the following properties:
\begin{enumerate}
\item For $L \not = H \oplus D_8 \oplus D_4$ the sextic curve $S$ is irreducible and $\mathcal{S}$ has a singularity of type $\mathbf{A}_{2n+1}$ at  $\mathrm{p}:  [u: v: w : y]= [0 : 0: 1: 0]$.
\item For $L = H \oplus D_8 \oplus D_4$ the sextic curve $S$ is the union of a line and an irreducible quintic and $\mathcal{S}$ has a singularity of type ${\bf A}_9$ at $\mathrm{p}$.
\item $\imath_\mathcal{S}, \jmath_\mathcal{S}$ induce the van Geemen-Sarti involution $\imath_\mathcal{X}$ and hyperelliptic involution $\jmath_\mathcal{X}$ of the alternate fibration, respectively.
\end{enumerate}
\end{proposition}
\begin{proof}
For $\mathcal{X}$ given by the alternate fibration~(\ref{eqn:vgs_intro}), the fiberwise translation by the 2-torsion section determines a van Geemen-Sarti involution $\imath_{\mathcal{X}}$ as explained in Section~\ref{sec:statment}. The minimal resolution of the quotient $\mathcal{X}/\langle \imath_{\mathcal{X}} \rangle$ gives a new K3 surface $\mathcal{Y}$ that supports again an alternate fibration. Following \cite{MR2274533}, a Weierstrass model for this alternate fibration is as follows:
\beq
\label{eqn:vgs_dual_11}
 \mathcal{Y}\colon \quad Y^2 Z = X \big( X^2 - a_4(u, v) \, X Z+  v^2 b''_6(u, v) \, Z^2 \big) \,.
\eeq
Then, $\mathcal{S}$ is obtained as a double cover of $\mathcal{Y}$ using the affine chart $Z=1$ and $X=v^2 w^2$, $Y=v^2 w y$. It follows that $\mathcal{S}$ is elliptically fibered over $\mathbb{P}(u, v)$ and its relative Jacobian fibration is easily checked to be isomorphic to $\mathcal{X}$. Moreover, the elliptic fibration admits sections. Hence, it is birational to $\mathcal{X}$. The explicit birational transformations between $\mathcal{X}$ and $\mathcal{S}$ are
\beq
\label{eqn:birational1}
 w = \frac{Y}{vX}\, , \qquad y=\frac{Y^2}{vX^2} - \frac{2 X + a_4(u, v) \, Z}{v Z} \, ,
\eeq
and
\beq
\label{eqn:birational2}
  X = -   \frac{v y + a_4(u, v) - v^2 w^2}{2}  Z \, , \qquad Y = -  \frac{ v^2 w y + a_4(u, v) \,  v w  - v^3 w^3}{2} Z \, ,
\eeq
respectively. For $L \not = H \oplus D_8 \oplus D_4$, one checks that the double sextic is irreducible and has exactly one singularity  at $[u: v: w: y]=[0: 0: 1: 0]$. This proves~(1). 
\par For $L = H \oplus D_8 \oplus D_4$ one can write $a_4(u, v)= v^2 \tilde{a}_2(u, v)$, $b''_6(u, v)= v \tilde{b}''_5(u, v)$, and the double sextic becomes
\beq
\label{eqn:double_sextic_rank_ge14a}
  \mathcal{S}\colon \quad y^2   = v \big( v w^4 - v \tilde{a}_2(u, v) \,  w^2 +   \tilde{b}''_5(u, v) \big)\,.
\eeq
Hence, the branch curve is the union of the line $v=0$ and an irreducible quintic. One easily checks that the double sextic $\mathcal{S}$ has a singularity of type ${\bf A}_9$  at $[u: v: w : y]=[0:0:1:0]$. This proves~(2).
\par For (3) the birational map in Equation~(\ref{eqn:birational1}) and~(\ref{eqn:birational2}) transforms the hyperelliptic involution $\jmath_\mathcal{X}$  to the involution $\jmath_\mathcal{S}: [u: v: w: y ] \mapsto  [u: v: -w: y]$, and the van Geemen-Sarti involution $\imath_\mathcal{X}$ to the involution $\imath_\mathcal{S}: [u: v: w: y ] \mapsto  [u: v: -w: -y ]$.  Lastly,  the involution $k_\mathcal{X}= \imath_\mathcal{X} \circ \jmath_\mathcal{X}$ is mapped to $k_\mathcal{S}= \imath_\mathcal{S} \circ \jmath_\mathcal{S}: [u: v: w: y ] \mapsto  [u: v: w: -y]$. 
\end{proof}
\par Next, we consider double sextics of the form given by Equation~(\ref{eqn:double-sextics_prime_general_intro}). For Picard number greater than 10, it is sufficient to consider a branch locus, that is the union of a uninodal quartic curve and a reducible conic splitting into two lines $\ell_1, \ell_2$.  After a suitable shift in the coordinates $v, w$, this is precisely $\mathcal{S}'$ in Equation~(\ref{eqn:double-sextics_prime_general_intro}) with
\beq
\label{eqn:double_sextic11_polys}
  C = w  \big( v + h_0 w  \big) \, , \qquad  Q= c_2(u, v) \, w^2 + e_3(u, v) \, w + d_4(u, v) \,.
\eeq
Here, $h_0 \in \mathbb{C}^\times$ and $c_2, e_3, d_4$ are homogeneous polynomials of degree 2, 3, and 4, respectively. Thus, the branch locus $S'$ has three irreducible components: the quartic curve $Q=0$ with node $\mathrm{n} = [0: 0: 1]$, and lines $\ell_1 = \mathrm{V}(v + h_0 w), \ell_2 = \mathrm{V} (w)$ which are not coincident with $\mathrm{n}$ (for $h_0 \neq 0$) and satisfy $\ell_1 \cap \ell_2 \cap Q = \emptyset$. 
\par We have the following:
\begin{proposition}
\label{prop:double_surface_Sp_rank11ff} 
For every double sextic $\mathcal{S}$ in Proposition~\ref{prop:double_surface_S_rank11ff} there are $h_0 \in \mathbb{C}^\times$ and polynomials $c_2, e_3, d_4$ such that  the double sextic
\beq
\label{eqn:double-sextics_prime_general_rank_ge11}
\mathcal{S}'\colon \quad \tilde{y}^2 =  w  \big( v + h_0 w  \big)  \big( c_2(u, v) \, w^2 + e_3(u, v) \, w + d_4(u, v) \big)
\eeq
is birational to $\mathcal{S}$.
\end{proposition}
\begin{remark}
In  \cite{Clingher:2022} it was proved that the pencil of lines through $\mathrm{n}$ induces the alternate fibration~(\ref{eqn:vgs_intro}) with
\beq
\label{eqn:params_rank11} 
  a_4 = v e_3(u, v) - 2  h_0 d_4(u, v)\, , \; b_8(u, v)=\frac{1}{4} a_4(u, v)^2 -  v^2 \left(\frac{e_3(u, v)^2}{4} -  c_2(u, v)  d_4(u, v) \right)  .
\eeq
\end{remark}
\begin{proof}
Let us first explain how one constructs the polynomials $c_2, e_3, d_4$ for the general $H\oplus D_4 \oplus A_1^{\oplus 5}$-polarized K3 surface $\mathcal{X}$. As explained before, we have $a_4(u, v)^2/4 - b_8(u, v) = v^2 b''_6(u, v)$ where $b''_6$ is a homogeneous polynomial of degree 6. We choose a factorization of $b_8(u, v) = a_4(u, v)^2/4 - v^2 b''_6(u, v)= d_4(u, v) d'_4(u, v)$ into two homogeneous polynomials of degree 4 such that $d(1,0) \not =0$.  We can then find $h_0 \in \mathbb{C^\times}$ so that 
\beq
\label{eqn:params1}
 e_3(u, v) =\frac{a_4(u, v) + 2 h_0 d_4(u, v)}{v}\, , \quad c_2(u, v) = \frac{d'_4(u, v)  + h_0^2 d_4(u, v) +h_0 a_4(u, v)}{v^2}
\eeq
are polynomials of degree 3 and 2, respectively.  In fact, if we write $a_4 = \sum_{n=0}^4 \alpha_n u^{4-n} v^{n}$, $d_4 = \sum_{n=0}^4  \delta_n u^{4-n} v^{n}$ (with $\delta_0 \not =0$), and $d'_4 = \sum_{n=0}^4  \delta'_n u^{4-n} v^{n}$, we have
\beq
\label{eqn:params2}
 \delta'_0 = \frac{\alpha_0}{4 \delta_0} \,, \quad \delta'_1 = \frac{\alpha_0 (2\alpha_1 \delta_0 - \alpha_0 \delta_1)}{4 \delta_0^2} \,, \quad  h_0 = - \frac{\alpha_0}{2 \delta_0} .
\eeq 
A surface $\mathcal{S}'$ is then obtained by using $h_0, c_2, e_3, d_4$ in Equation~(\ref{eqn:double_sextic11_polys}).  Moreover, the case $h_0=0$ is equivalent to $\alpha_0=0$ and $d'_4 = v^2 c_2$, so that we have $a_4=v e_3$ and $b_8=v^2 c_2 d_4$. The explicit birational transformations between $\mathcal{X}$ and $\mathcal{S}'$ are
\beq
\label{eqn:birational1p}
w = \frac{vX}{(c_2 v^2 - h_0 e_3 v + h_0^2 d_4) Z - h_0X}\, , \qquad \tilde{y}= \frac{v (c_2 v^2 - h_0 e_3 v + h_0^2 d_4)  YZ}{((c_2 v^2 - h_0 e_3 v + h_0^2 d_4) Z - h_0X)^2}\, ,
\eeq
and
\beq
\label{eqn:birational2p}
X =  \frac{ (c_2 v^2 - h_0 e_3 v + h_0^2 d_4) w}{v + h_0w} Z\, , \qquad Y= \frac{ (c_2 v^2 - h_0 e_3 v + h_0^2 d_4) v \tilde{y}}{(v + h_0w)^2} Z\,,
\eeq
respectively. 
\par The lattice polarization extends to the lattice $H\oplus D_6 \oplus A_1^{\oplus 4}$ if and only if $a_4(u, v)=v e_3(u, v)$ and $b_8(u, v)=v^2 c_2(u, v) d_4(u, v)$. For any other decomposition $b_8(u, v) = v^2 \tilde{c}_2(u, v) \tilde{d}_4(u, v)$ the map $[ \tilde{u}: \tilde{v}: \tilde{w}: \tilde{\tilde{y}}] = [   \tilde{c}_2(u, v) u:  \ \tilde{c}_2(u, v) v : \  c_2(u, v) w :\  c_2(u, v) \tilde{c}_2(u, v)^2 \tilde{y}  ]$ transforms the equation
\beq
    \tilde{\tilde{y}}^2 =  \tilde{w}  \tilde{v}  \Big( \tilde{c}_2(\tilde{u}, \tilde{v}) \, \tilde{w}^2 + e_3(\tilde{u}, \tilde{v}) \, \tilde{w} + \tilde{d}_4(\tilde{u}, \tilde{v}) \Big) \,.
\eeq
into $\mathcal{S}'$ in Equation~(\ref{eqn:double-sextics_prime_general_intro}) using the polynomials in Equation~(\ref{eqn:double_sextic11_polys}) with $h_0=0$.  Moreover, Equations~(\ref{eqn:params1}) and~(\ref{eqn:params2}) and $b = d_4 \cdot d'_4$ determine $h_0 \in \mathbb{C}^\times$ and polynomials $c_2, e_3, d_4$ for $\rho_L = 11, 12, 13$ and $L= H \oplus D_8 \oplus D_4$. For $L= H \oplus E_8 \oplus A_1^{\oplus 4}$ and $\rho_L=15, 16, 17, 18$ we write $a_4=v e_3$ and $b_8=v^2 c_2 d_4$.  In these cases the Picard number is $\rho_L = 10+n$, $n=4, \dots, 8$, and the vanishing degree of $b_8$ at $v=0$ is $n$. One then has to make some additional choices: We distribute the repeated root of $b_8$ at $v=0$ between the polynomials $c_2$ and $d_4$ as follows:
\beqn
 \begin{array}{c|ccc}
 \text{rank} & c_2 & d_4 & b\\
  \hline
   14	& O(v^0) & O(v^2) & O(v^4)\\
 15 	& O(v^0) & O(v^3) & O(v^5)\\
 16	& O(v^1) & O(v^3) & O(v^6)\\
 17	& O(v^1) & O(v^4) & O(v^7)\\
 18	& O(v^2) & O(v^4) & O(v^8)\\
 \hline
 \end{array}
\eeqn
By using these polynomials $c_2, e_3, d_4$ and $h_0=0$ we obtain a birational double sextic $\mathcal{S}'$ as claimed.
\end{proof}
\par For $\mathcal{S}'$ in Proposition~\ref{prop:double_surface_Sp_rank11ff} we consider the associated quartic hypersurface $\mathcal{K}$ in Equation~(\ref{eqn:quartic_general_intro}) using the same polynomials in~(\ref{eqn:double_sextic11_polys}).  We have the following:
\begin{corollary}
\label{prop:regular}
Let $\mathcal{X}$ be a general K3 surface with $\mathrm{NS}(\mathcal{X}) \cong L$, for $L$ in Equation~(\ref{eqn:lattices}) with $\rho_L =10+n$, $1 \le n \le 8$. There are $h_0 \in \mathbb{C}^\times$ and polynomials $c_2, e_3, d_4$ such that $\mathcal{X}$ is isomorphic to the minimal resolution of the quartic $\mathcal{K}$ given by
\beq
\label{eqn:quartic_general_rank_ge11}
\mathcal{K}\colon \quad  w  \big( v + h_0 w  \big)  y^2 =  c_2(u, v) \, w^2 + e_3(u, v) \, w + d_4(u, v) \, .
\eeq
Moreover, $\mathcal{K}$ has two rational double-point singularities at
\beq
\label{eqn:singular_points}
   \mathrm{p}_1\colon \ [u: v: w : y]= [0: 0: 0: 1] \,, \qquad   \mathrm{p}_2\colon \  [u: v: w : y]= [0 : 0: 1: 0] \,,
\eeq  
of the type indicated in Table~\ref{tab:sings_on_K_and_S}.
\end{corollary}
\begin{remark}
In Table~\ref{tab:sings_on_K_and_S} we show the types of singularities determined in Corollary~\ref{prop:regular}. We also include the results of Lemmas~\ref{lem:sings_on_S_rank10},~\ref{lem:sings_K_rank10}.

\end{remark}
\begin{table}
\begin{tabular}{l|c||c||cc}
$\rho_L$	& $(\ell_L, \delta_L)$ & singularity on $\mathcal{S}$	& \multicolumn{2}{l}{singularities on $\mathcal{K}$} \\
&& at $\mathrm{p}$			& \quad at $\mathrm{p}_2$	& at $\mathrm{p}_1$ \\
\hline
\hline
 9		& $(9,1)$ & none				&  \multicolumn{2}{l}{$\mathcal{S}$ and $\mathcal{K}$ not isomorphic} \\
 9		& $(7,1)$ &n/a					& --					& ${\bf A}_1$ \\
 \hline
10		& $(8,1)$& ${\bf A}_1$ 			& ${\bf A}_1$ 			& ${\bf A}_1$ \\
 \hline
11		& $(7,1)$& ${\bf A}_3$			& ${\bf A}_1$ 			& ${\bf A}_3$ \\
12		& $(6,1)$& ${\bf A}_5$			& ${\bf A}_3$ 			& ${\bf A}_3$ \\
13		& $(5,1)$& ${\bf A}_7$			& ${\bf A}_3$ 			& ${\bf A}_5$ \\
14		& $(4,0)$& ${\bf A}_9$			& ${\bf A}_3$ 			& ${\bf A}_7$ \\
14		& $(4,1)$& ${\bf A}_9$			& ${\bf A}_3$ 			& ${\bf A}_7$ \\
15		& $(3,1)$& ${\bf A}_{11}$			& ${\bf A}_3$ 			& ${\bf A}_9$ \\
16		& $(2,1)$& ${\bf A}_{13}$			& ${\bf A}_5$ 			& ${\bf A}_9$ \\
17		& $(1,1)$& ${\bf A}_{15}$			& ${\bf A}_5$ 			& ${\bf A}_{11}$ \\
18		& $(0,0)$ & ${\bf A}_{17}$			& ${\bf E}_6$ 			& ${\bf A}_{11}$\\
\hline
\end{tabular}
\captionsetup{justification=centering}
\caption{Rational double points on $\mathcal{S}$ and  $\mathcal{K}$}
\label{tab:sings_on_K_and_S}
\end{table}
\begin{proof}
Because of Proposition~\ref{prop:double_surface_S_rank11ff}, there is a double sextic $\mathcal{S}$ whose minimal resolution is isomorphic to $\mathcal{X}$. As explained above we constructed from $\mathcal{S}$ a birational double sextic $\mathcal{S}'$. It follows that the minimal resolution of the associated quartic $\mathcal{K}$ is isomorphic to $\mathcal{X}$. In particular, $\mathcal{K}$ only has rational double points. An explicit computation shows that these are the ones listed in Table~\ref{tab:sings_on_K_and_S}. 
\end{proof}
For every lattice $L$ in Equation~(\ref{eqn:lattices}) with $\rho_L =10+n$, $1 \le n \le 8$, the dual graph of smooth rational curves on a general K3 surface $\mathcal{X}$ with $\mathrm{NS}(\mathcal{X}) = L$ is shown in Figures~\ref{fig:pic11}-\ref{fig:pic18}. They were constructed by the third author in \cite{Roulleau22}. Note that for simplicity for rank $\rho_L < 14$ the graphs do not show \emph{all} rational curves on $\mathcal{X}$.  The complete graphs can be found in \cite{Clingher:2022}. However, each graph shows a complete set of generators of $\mathrm{NS}(\mathcal{X})$, denoted by $\{ A_n \}$.  For each lattice $L$ a divisor  $\mathpzc{D}_2 \in \mathrm{NS}(\mathcal{X})$ is defined in Table~\ref{tab:D2_divisors} in terms of these generators; similarly, divisors $\{ \mathrm{a}_1, \dots, \mathrm{a}_M \}$ and $\{ \mathrm{b}_1, \dots, \mathrm{b}_N \}$ are defined in Table~\ref{tab:except_divisors}. (The hat symbol means that the particular element is not included in a sequence.)  In Tables~\ref{tab:D2_divisors} and~\ref{tab:except_divisors} we also included the results of Propositions~\ref{prop:symmetry_rank10b} and~\ref{prop_divisor_rank10}.
\begin{table}
\begin{tabular}{l|c||l}
$\rho_L$ 	& $(\ell_L, \delta_L)$ &  $\mathpzc{D}_2$ \\
\hline
\hline
10	& $(8,1)$ &$3 \mathrm{f}_1 + 3 \mathrm{f}_2 - \sum_{i=1}^8 \mathrm{e}_i$\\
\hline
11	&  $(7,1)$	&$A_{1}+2A_{2}+A_{3}+A_{4}$\\
12	&  $(6,1)$	&$A_6 + A_7 + 3 A_8 + 2 A_9 + 2 A_{10} + A_{11} + A_{12}$\\
13	&  $(5,1)$	&$A_4 + A_6 + A_7 + A_8 + A_{10} + A_{11} + A_{12} + A_{16} + A_{18}$\\
14	&  $(4,1)$	&$A_1 + 2 A_2 + 3 A_3 + 4 A_4 + 2 A_5 + 5 A_6 + 4 A_7 + 3 A_8 + 2 A_9 + A_{10}$\\
14  	&  $(4,0)$	&$A_1 + 2 A_2 + 3 A_3 + A_4 + 3 A_5 + \dots + 3 A_8 + A_9 + 2 A_{10} + A_{11}$\\
15	&  $(3,1)$	&$A_1 + 2 A_2 + 3 A_3 + A_4 + 3 A_5 + \dots +  3 A_{10} + A_{11} + 2 A_{12} + A_{13}$ \\
16	&  $(2,1)$	&$A_1 + 2 A_2 + 3 A_3 + A_4 + 3 A_5 + \dots + 3 A_{12} + A_{13} + 2 A_{14} + A_{15}$ \\
17	&  $(1,1)$	&$A_2 + \dots + A_{17} + A_{19}$ \\
18    &  $(0,0)$	&$A_1 + 2 A_2 + 3 A_3 + A_4 + 3 A_5 + \dots + 3 A_{15} + A_{16} + 2 A_{17} + A_{18} + 3 A_{19}$\\
\hline
\end{tabular}
\captionsetup{justification=centering}
\caption{Nef divisors for $\mathcal{X}$ for $\mathrm{NS}(\mathcal{X}) = L$}
\label{tab:D2_divisors}
\end{table}
\begin{table}
\begin{tabular}{l|c||l|l|l}
$\rho_L$ 	& $(\ell_L, \delta_L)$ & $(N, M)$ &  $\{\mathrm{b}_1, \dots, \mathrm{b}_N\}$ & $\{\mathrm{a}_1, \dots, \mathrm{a}_M\}$\\
\hline
\hline
10	& $(8,1)$ & $(1,1)$ & $\{\mathrm{e}_8\}$					& $\{- \mathrm{f}_1 + \mathrm{f}_2\}$\\
\hline
11	&  $(7,1)$	& $(1,1)$ &$\{2 A_1 - A_4 + A_5 + \dots +A_8\}$	& $\{A_1 \}$\\
12	&  $(6,1)$	& $(3,3)$ &$\{A_{11}, A_{18}, A_{19}\}$			& $\{A_8, A_9, A_{12}\}$\\
13	&  $(5,1)$	& $(3,5)$ &$\{A_{12}, A_{13}, A_{14}\}$			& $\{A_6, \dots, \widehat{A_9}, \dots, A_{11}\}$\\
14	&  $(4,0)$	& $(3,7)$ &$\{A_{10}, A_{11}, A_{12}\}$			& $\{A_1, \dots, \widehat{A_5}, \dots, A_8\}$\\
14	&  $(4,1)$ & $(3,7)$ &$\{A_{11}, A_{12}, A_{15}\}$			& $\{A_1, \dots, \widehat{A_4}, \dots, A_8\}$\\
15	&  $(3,1)$	& $(3,9)$ &$\{A_{13}, A_{14}, A_{16}\}$			& $\{A_1, \dots, \widehat{A_4}, \dots, A_{10}\}$\\
16	&  $(2,1)$	& $(5,9)$ 	&$\{A_{12}, \dots , A_{16}\}$			& $\{A_1, \dots, \widehat{A_4}, \dots, A_{10}\}$\\
17	&  $(1,1)$	& $(5,11)$ &$\{A_1, \dots, A_5\}$				& $\{A_7, \dots, A_{16}, A_{19}\}$\\
18    &  $(0,0)$	& $(6,11)$ &$\{A_{13}, \dots, A_{18}\}$			& $\{A_1, \dots, \widehat{A_4}, \dots, A_{11}, A_{19}\}$ \\
\hline
\end{tabular}
\captionsetup{justification=centering}
\caption{Exceptional divisors for the singularity at $\mathrm{p}_2$ and  $\mathrm{p}_1$ on $\mathcal{K}$}
\label{tab:except_divisors}
\end{table}
\newpage
\par We have the following:
\begin{theorem} \leavevmode
\label{thm_rank_ge11}
\begin{enumerate}
\item In Table~\ref{tab:D2_divisors}  each divisor $\mathpzc{D}_2 \in \mathrm{NS}(\mathcal{X})$ is nef,  base-point free, of square 2, and invariant under $\imath_\mathcal{X}$. 
\item In Table~\ref{tab:except_divisors} the divisors $\mathrm{a}_1, \dots, \mathrm{a}_M$ and $\mathrm{b}_1, \dots, \mathrm{b}_N  \in \mathrm{NS}(\mathcal{X})$ are exceptional divisors for the singularity at $\mathrm{p}_1$ and $\mathrm{p}_2$ on $\mathcal{K}$ in Equation~(\ref{eqn:quartic_general_rank_ge11}).
\item The polarization divisor of $\mathcal{K}$ in Equation~(\ref{eqn:quartic_general_rank_ge11}) is $\mathpzc{H} = \mathpzc{D}_2 + \mathrm{b}_1 + \dots + \mathrm{b}_N$.
\item  After a suitable choice of coordinates, for the linear system $\vert \mathpzc{D}_2 \vert$ the corresponding rational map is the projection onto $\mathbb{P}(u, v, w)$ in Equation~(\ref{eqn:double-sextics_general_rang_ge11}).
\item  $\mathpzc{D}'_2 = \mathpzc{H} - \mathrm{a}_1 - \dots - \mathrm{a}_M \in \mathrm{NS}(\mathcal{X})$ is nef and of square 2. After a suitable choice of coordinates, for the linear system $|\mathpzc{D}'_2|$ the corresponding rational map is the projection onto $\mathbb{P}(u, v, w)$ in Equation~(\ref{eqn:double-sextics_prime_general_rank_ge11}).
\end{enumerate}
\end{theorem}
\begin{remark}
For $\rho_L = 18$ the singularity at $\mathrm{p}_2$ is of type ${\bf E}_6$, and not of type ${\bf A}_n$ as in all other cases. In this case, the precise relation between divisors in Theorem~\ref{thm_rank_ge11}~(3) is $\mathpzc{H} = \mathpzc{D}_2 + \mathrm{b}_1 + 2 \mathrm{b}_2 + 3 \mathrm{b}_3 + 2 \mathrm{b}_4 + 2 \mathrm{b}_5 + \mathrm{b}_6$ for $\mathrm{b}_i = A_{12+i}$ with $i=1, \dots, 6$.
\end{remark}
\begin{remark}
The divisor $\mathpzc{D}'_2$ is not invariant  under $\imath_\mathcal{X}$. 
\end{remark}
\begin{proof}
To prove claim~(1) we note that the divisors $\mathpzc{D}_2$ in Table~\ref{tab:D2_divisors} were constructed by Roulleau in \cite{Roulleau22} who also proved that they are nef,  base-point free, of square 2. Clingher and Malmendier~\cite{Clingher:2022} constructed the unique embedding of the alternate fibration into the dual graphs and the action of the van Geemen-Sarti involution.  One then easily checks that the divisors $\mathpzc{D}_2$ in Table~\ref{tab:D2_divisors} are invariant under this action. To prove claims~(2) and~(3), we note that the divisors $\mathrm{a}_1, \dots, \mathrm{a}_M$ and $\mathrm{b}_1, \dots, \mathrm{b}_N$ and $\mathcal{K}$ were identified in \cites{Clingher:2022,Clingher:2020baq, MR4544843, MR4160930}. Comparing these divisors with the labels in Figures~\ref{fig:pic11}-\ref{fig:pic18} yields the statement.
\par Claim~(4) follows from the following geometric argument: Let $\mathcal{K}$ be the quartic surface in $\mathbb{P}^3$, with two rational double point singularities at $\mathrm{p}_1$ and $\mathrm{p}_2$. Let $\mathcal{X}$ be the smooth K3 surface obtained after one performs desingularization at $\mathrm{p}_1$ and $\mathrm{p}_2$, with $\psi \colon \mathcal{X} \rightarrow \mathbb{P}^3$ being the associated desingularization map. Consider then the projection from $\mathrm{p}_2$. This is a rational map $\pi \colon \mathbb{P}^3 \dashrightarrow \mathbb{P}^2 $. The target projective space may be seen as the space of lines in $\mathbb{P}^3$ passing through $\mathrm{p}_2$. The restriction of $\pi$ on $\mathcal{K} - \{\mathrm{p}_2\} $ is a morphism and the generic fiber of this restriction is given by the two points obtained as the residual intersection of the associated line through $\mathrm{p}_2$ with $\mathcal{K} - \{\mathrm{p}_2\} $.  The composition $\psi \circ \left ( \pi \vert _{\mathcal{K} - \mathrm{p}_2} \right ) $ extends then to a morphism $\phi \colon  \mathcal{X} \rightarrow \mathbb{P}^2 $ which recovers the canonical double sextic structure.    
\beq
\label{diaggl7}
\def\objectstyle{\scriptstyle}
\def\labelstyle{\scriptstyle}
\xymatrix @-0.9pc  {
 \mathcal{X} \ar @{->} [rr] ^{\psi} \ar @{->} [rrdd] _{\phi}  & & \mathbb{P}^3 \ar @{-->} [dd] ^{\pi} \\
  & & \\
    & & \mathbb{P}^2  \\
} 
\eeq
The branch locus of $\phi$ corresponds to the compactification of the set of lines through $\mathrm{p}_2$ that either pass through $\mathrm{p}_1$ or are tangent to $\mathcal{K} - \{ \mathrm{p}_2\} $. 
\par Now, in the context of the diagram~(\ref{diaggl7}), the linear system associated to the map $\psi$ is $\vert \mathpzc{H} \vert $, and consists of the pullbacks $ \psi^*(H)$ where $H$ is a plane in $\mathbb{P}^3$.  A generic member of 
$\vert \mathpzc{H} \vert $ is a plane degree-four curve obtained by intersection a generic plane of $\mathbb{P}^2$ with $\mathcal{K}$. Special members of $\vert \mathpzc{H} \vert $ can be obtained by considering $ \psi^*(H)$, where $H$ is a plane in $\mathbb{P}^3$ passing through $\mathrm{p}_2$. Such a plane corresponds, in the context of the diagram, to a line in $\mathbb{P}^2$, and one has $\psi^*(H) = \phi^*(\ell)$. Consider a generic line $\ell$ in $\mathbb{P}^2$. Then $ \left ( \pi \vert _{\mathcal{K} - \mathrm{p}_2} \right )^*(\ell) $ corresponds to a plane curve $Q$ lying on $\mathcal{K}$ and $\vert \psi^*(Q) \vert $ is the linear system associated to the map $\phi \colon \mathcal{X} \rightarrow \mathbb{P}^2$. Moreover, via the desingularization map $\psi$, one obtains: 
\beqn
  \phi^*(\ell) = \psi^*(Q) + \mathrm{b}_1 + \dots + \mathrm{b}_N \, . 
\eeqn  
Therefore, we get that $\psi^*(Q) + \mathrm{b}_1 + \dots + \mathrm{b}_N $ is a special member of $ \vert  \mathpzc{H}  \vert $ and hence:
\beqn
  \mathpzc{H}  \ \sim \ \psi^*(Q) + \mathrm{b}_1 + \dots + \mathrm{b}_N   \,  .
\eeqn  
Therefore, we obtain
\beqn
 \mathpzc{D}_2 =  \mathpzc{H} - \left ( \mathrm{b}_1 + \dots + \mathrm{b}_N \right ) \  \sim \ \psi^*(Q) \, . 
\eeqn
It follows that $\vert  \mathpzc{D}_2 \vert $ is the linear system associated with the morphism $\phi$.   The first part of~(5) follows by standard lattice theoretic computations. The second part is analogous to~(4).
\end{proof}
\begin{figure}
  	\centering
	\scalebox{0.7}{%
    		\begin{tikzpicture}[rotate=90]
       		\draw [domain=-5:0] plot(\x,{0.67*(5-abs(\x))^0.5});
\draw [domain=-5:0] plot(\x,{-0.67*(5-abs(\x))^0.5});

\draw [very thick] (-6,0) -- (-5,0);

\draw (0,1.5) -- (3.5,0.5);
\draw (0,1.5) -- (2.5,0.5);
\draw (0,1.5) -- (1.5,0.5);
\draw (0,1.5) -- (0.5,0.5);
\draw (0,1.5) -- (-0.5,0.5);
\draw (0,1.5) -- (-1.5,0.5);
\draw (0,1.5) -- (-2.5,0.5);
\draw (0,1.5) -- (-3.5,0.5);

\draw [very thick] (3.5,0.5) -- (3.5,-0.5);
\draw [very thick] (2.5,0.5) -- (2.5,-0.5);
\draw [very thick] (1.5,0.5) -- (1.5,-0.5);
\draw [very thick] (0.5,0.5) -- (0.5,-0.5);
\draw [very thick] (-3.5,0.5) -- (-3.5,-0.5);
\draw [very thick] (-2.5,0.5) -- (-2.5,-0.5);
\draw [very thick] (-1.5,0.5) -- (-1.5,-0.5);
\draw [very thick] (-0.5,0.5) -- (-0.5,-0.5);

\draw (0,-1.5) -- (3.5,-0.5);
\draw (0,-1.5) -- (2.5,-0.5);
\draw (0,-1.5) -- (1.5,-0.5);
\draw (0,-1.5) -- (0.5,-0.5);
\draw (0,-1.5) -- (-0.5,-0.5);
\draw (0,-1.5) -- (-1.5,-0.5);
\draw (0,-1.5) -- (-2.5,-0.5);
\draw (0,-1.5) -- (-3.5,-0.5);

\draw (0,1.5) node {$\bullet$};
\draw (3.5,0.5) node  {$\bullet$};
\draw (2.5,0.5) node  {$\bullet$};
\draw (1.5,0.5) node  {$\bullet$};
\draw (0.5,0.5) node  {$\bullet$};
\draw (-0.5,0.5) node  {$\bullet$};
\draw (-1.5,0.5) node  {$\bullet$};
\draw (-2.5,0.5) node  {$\bullet$};
\draw (-3.5,0.5) node  {$\bullet$};

\draw (-5,0) node  {$\bullet$};
\draw (-6,0) node  {$\bullet$};

\draw (-0.3,1.8) node [above]{$A_{1}$};
\draw (3.5,0.5) node [left]{$A_{12}$};
\draw (2.5,0.5) node [left]{$A_{11}$};
\draw (1.5,0.5) node [left]{$A_{10}$};
\draw (0.5,0.5) node [left]{$A_{9}$};
\draw (-0.5,0.5) node [left]{$A_{8}$};
\draw (-1.5,0.5) node [left]{$A_{7}$};
\draw (-2.5,0.5) node [left]{$A_{6}$};
\draw (-3.5,0.5) node [left]{$A_{5}$};

\draw (0,-1.5) node {$\bullet$};
\draw (3.5,-0.5) node  {$\bullet$};
\draw (2.5,-0.5) node  {$\bullet$};
\draw (1.5,-0.5) node  {$\bullet$};
\draw (0.5,-0.5) node  {$\bullet$};
\draw (-0.5,-0.5) node  {$\bullet$};
\draw (-1.5,-0.5) node  {$\bullet$};
\draw (-2.5,-0.5) node  {$\bullet$};
\draw (-3.5,-0.5) node  {$\bullet$};

\draw (0.3,-1.8) node [below]{$A_{3}$};
\draw (3.5,-1.4) node [left]{$A_{20}$};
\draw (2.5,-1.4) node [left]{$A_{19}$};
\draw (1.5,-1.4) node [left]{$A_{18}$};
\draw (0.5,-1.4) node [left]{$A_{17}$};
\draw (-0.5,-1.4) node [left]{$A_{16}$};
\draw (-1.5,-1.4) node [left]{$A_{15}$};
\draw (-2.5,-1.4) node [left]{$A_{14}$};
\draw (-3.5,-1.4) node [left]{$A_{13}$};
\draw (-5.6,0.4) node [below]{$A_{4}$};
\draw (-4.3,0) node [below]{$A_{2}$};

    		\end{tikzpicture}}
\caption{Dual graph for $\mathrm{NS}(\mathcal{X}) = H\oplus D_4 \oplus A_1^{\oplus 5}$}
 \label{fig:pic11}
\end{figure}
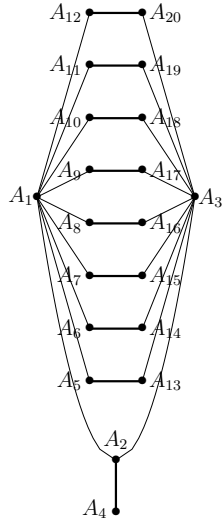
\begin{figure}
  	\centering
	\scalebox{0.7}{%
    		\begin{tikzpicture}[rotate=90]
       		\draw (0,1.5) -- (2.5,0.5);
\draw (0,1.5) -- (1.5,0.5);
\draw (0,1.5) -- (0.5,0.5);
\draw (0,1.5) -- (-0.5,0.5);
\draw (0,1.5) -- (-1.5,0.5);
\draw (0,1.5) -- (-2.5,0.5);

\draw [very thick] (2.5,0.5) -- (2.5,-0.5);
\draw [very thick] (1.5,0.5) -- (1.5,-0.5);
\draw [very thick] (0.5,0.5) -- (0.5,-0.5);
\draw [very thick] (-2.5,0.5) -- (-2.5,-0.5);
\draw [very thick] (-1.5,0.5) -- (-1.5,-0.5);
\draw [very thick] (-0.5,0.5) -- (-0.5,-0.5);

\draw (-3.5,0) -- (-3.5-0.866,0.5);
\draw (-3.5,0) -- (-3.5-0.866,-0.5);
\draw (0,1.5) -- (-3.5,1.5);
\draw (-3.5,1.5) -- (-3.5,-1.5);
\draw (0,-1.5) -- (-3.5,-1.5);
\draw (0,-1.5) -- (2.5,-0.5);
\draw (0,-1.5) -- (1.5,-0.5);
\draw (0,-1.5) -- (0.5,-0.5);
\draw (0,-1.5) -- (-0.5,-0.5);
\draw (0,-1.5) -- (-1.5,-0.5);
\draw (0,-1.5) -- (-2.5,-0.5);

\draw (-3.5-0.866,0.5) node {$\bullet$};
\draw (-3.5-0.866,-0.5) node {$\bullet$};
\draw (0,1.5) node {$\bullet$};
\draw (-3.5,1.5) node  {$\bullet$};
\draw (2.5,0.5) node  {$\bullet$};
\draw (1.5,0.5) node  {$\bullet$};
\draw (0.5,0.5) node  {$\bullet$};
\draw (-0.5,0.5) node  {$\bullet$};
\draw (-1.5,0.5) node  {$\bullet$};
\draw (-2.5,0.5) node  {$\bullet$};
\draw (-3.5,-1.5) node  {$\bullet$};

\draw (0,1.5) node [left]{$A_{ 12}$};
\draw (2.5,0.5) node [left]{$A_{ 13}$};
\draw (1.5,0.5) node [left]{$A_{ 5}$};
\draw (0.5,0.5) node [left]{$A_{ 4}$};
\draw (-0.5,0.5) node [left]{$A_{ 3}$};
\draw (-1.5,0.5) node [left]{$A_{ 2}$};
\draw (-2.4,0.5) node [left]{$A_{ 1}$};
\draw (-3.5,-1.5) node [right]{$A_{ 10}$};

\draw (0,-1.5) node {$\bullet$};
\draw (2.5,-0.5) node  {$\bullet$};
\draw (1.5,-0.5) node  {$\bullet$};
\draw (0.5,-0.5) node  {$\bullet$};
\draw (-0.5,-0.5) node  {$\bullet$};
\draw (-1.5,-0.5) node  {$\bullet$};
\draw (-2.5,-0.5) node  {$\bullet$};
\draw (-3.5,0) node  {$\bullet$};

\draw (0,-1.5) node [right]{$A_{ 11}$};
\draw (2.5,-1.4) node [left]{$A_{ 19}$};
\draw (1.5,-1.4) node [left]{$A_{ 18}$};
\draw (0.5,-1.4) node [left]{$A_{ 17}$};
\draw (-0.5,-1.4) node [left]{$A_{ 16}$};
\draw (-1.5,-1.4) node [left]{$A_{ 15}$};
\draw (-2.4,-1.4) node [left]{$A_{ 14}$};
\draw (-3.5,1.5) node [left]{$A_{ 9}$};

\draw (-3.5-0.866,0.5) node [left]{$A_{ 6}$};
\draw (-3.5-0.866,-0.5) node [right]{$A_{ 7}$};
\draw (-3.5,0) node [above]{$A_{ 8}$};

    		\end{tikzpicture}}
\caption{Dual graph for $\mathrm{NS}(\mathcal{X}) = H\oplus D_6 \oplus A_1^{\oplus 4}$}
 \label{fig:pic12}
\end{figure}
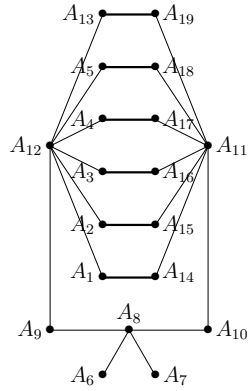
\begin{figure}
  	\centering
	\scalebox{0.7}{%
    		\begin{tikzpicture}[rotate=-90]
       		\draw (0,0) -- (-2,1);
\draw (0,0) -- (-1,1);
\draw (0,-1) -- (0,4);
\draw (0,0) -- (1,1);
\draw (0,0) -- (2,1);

\draw (0,4) -- (4,4);
\draw (3,4) -- (3,-1);
\draw (0,-1) -- (4,-1);

\draw (0,3) -- (-2,2);
\draw (0,3) -- (-1,2);
\draw (0,3) -- (1,2);
\draw (0,3) -- (2,2);

\draw [very thick] (-2,2) -- (-2,1);
\draw [very thick] (-1,2) -- (-1,1);
\draw [very thick] (0,1) -- (0,2);
\draw [very thick] (1,2) -- (1,1);
\draw [very thick] (2,2) -- (2,1);

\draw (0,0) node {$\bullet$};
\draw (-2,1) node {$\bullet$};
\draw (-1,1) node {$\bullet$};
\draw (0,1) node {$\bullet$};
\draw (1,1) node {$\bullet$};
\draw (2,1) node {$\bullet$};
\draw (-2,2) node {$\bullet$};
\draw (-1,2) node {$\bullet$};
\draw (0,2) node {$\bullet$};
\draw (1,2) node {$\bullet$};
\draw (2,2) node {$\bullet$};
\draw (0,3) node {$\bullet$};
\draw (0,4) node {$\bullet$};
\draw (3,4) node {$\bullet$};
\draw (4,4) node {$\bullet$};
\draw (3,1.5) node {$\bullet$};
\draw (3,-1) node {$\bullet$};
\draw (4,-1) node {$\bullet$};
\draw (0,-1) node {$\bullet$};

\draw (0,-0.2) node [above]{$A_{12}$};
\draw (-2,1) node [left]{$A_{13}$};
\draw (-1,1) node [left]{$A_{14}$};
\draw (0.1,1) node [left]{$A_{15}$};
\draw (1,1) node [left]{$A_{16}$};
\draw (2,1) node [left]{$A_{17}$};
\draw (-2,2) node [right]{$A_{1}$};
\draw (-1,2) node [right]{$A_{2}$};
\draw (0,2) node [right]{$A_{3}$};
\draw (1,2) node [right]{$A_{4}$};
\draw (2,2) node [right]{$A_{5}$};
\draw (0,3.2) node [above]{$A_{6}$};
\draw (0,4) node [right]{$A_{7}$};
\draw (3,4) node [right]{$A_{8}$};
\draw (4,4) node [right]{$A_{9}$};
\draw (3,1.5) node [below]{$A_{10}$};
\draw (3,-1) node [left]{$A_{11}$};
\draw (4,-1) node [left]{$A_{19}$};
\draw (0,-1) node [left]{$A_{18}$};
    		\end{tikzpicture}}
\caption{Dual graph for $\mathrm{NS}(\mathcal{X}) = H\oplus E_7 \oplus A_1^{\oplus 4}$}
 \label{fig:pic13}
\end{figure}
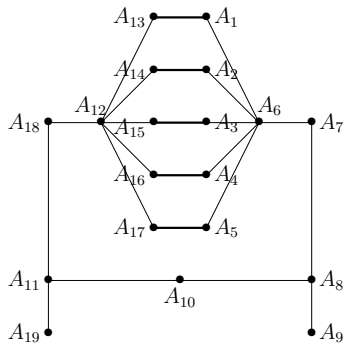
\begin{figure}
  	\centering
	\scalebox{0.7}{%
    		\begin{tikzpicture}
       		\draw (7,0) -- (12,0);
\draw (7,0) -- (7,-3);
\draw (7,-3) -- (12,-3);
\draw (12,0) -- (12,-3);
\draw [very thick] (9,0) -- (10,0);
\draw  (9,1) -- (8,0);
\draw  (9,-1) --  (8,0); 
\draw  (9,2) -- (8,0);
\draw  (9,-2) --  (8,0); 

\draw [very thick]  (9,2) -- (10,2);
\draw [very thick]  (9,1) -- (10,1);
\draw [very thick] (9,-1) --  (10,-1); 
\draw [very thick]  (9,-2) -- (10,-2);
\draw  (10,1) -- (11,0);
\draw  (10,-1) --  (11,0); 
\draw  (10,2) -- (11,0);
\draw  (10,-2) --  (11,0);

\draw (9.5,-3) -- (9.5,-4);

\draw (12,0) node {$\bullet$};
\draw (12,-1.5) node {$\bullet$};
\draw (12,-3) node {$\bullet$};
\draw (9.5,-3) node {$\bullet$};
\draw (7,-3) node {$\bullet$};
\draw (7,-1.5) node {$\bullet$};
\draw (7,0) node {$\bullet$};
\draw (8,0) node {$\bullet$};
\draw (9,0) node {$\bullet$};
\draw (10,0) node {$\bullet$};
\draw (11,0) node {$\bullet$};
\draw (9.5,-4) node {$\bullet$};
\draw (9,1) node {$\bullet$};
\draw (10,1) node {$\bullet$};
\draw (9,-1) node {$\bullet$};
\draw (10,-1) node {$\bullet$};
\draw (9,2) node {$\bullet$};
\draw (10,2) node {$\bullet$};
\draw (9,-2) node {$\bullet$};
\draw (10,-2) node {$\bullet$};

\draw (9.5,-4) node [right]{$A_{5}$};
\draw (12,0) node [right]{$A_{2}$};
\draw (12,-1.5) node [right]{$A_{3}$};
\draw (12,-3) node [right]{$A_{4}$};
\draw (9.5,-3) node [above]{$A_{6}$};
\draw (7,-3) node [left]{$A_{7}$};
\draw (7,-1.5) node [left]{$A_{8}$};
\draw (7,0) node [above]{$A_{9}$};
\draw (7.8,0) node [above]{$A_{10}$};
\draw (9,0) node [above]{$A_{13}$};
\draw (10,0) node [above]{$A_{18}$};
\draw (11,0) node [above right]{$A_{1}$};
\draw (9,1) node [above]{$A_{12}$};
\draw (10,1) node [above]{$A_{17}$};
\draw (9,-1) node [below]{$A_{14}$};
\draw (10,-1) node [below]{$A_{19}$};
\draw (9,2) node [above]{$A_{11}$};
\draw (10,2) node [above]{$A_{16}$};
\draw (9,-2) node [below]{$A_{15}$};
\draw (10,-2) node [below]{$A_{20}$};
    		\end{tikzpicture}}
\caption{Dual graph for $\mathrm{NS}(\mathcal{X}) = H\oplus D_8 \oplus D_4^{\oplus 4}$}
\label{fig:pic14p}
\end{figure}
\begin{figure}
  	\centering
	\scalebox{0.7}{%
    		\begin{tikzpicture}
       		\draw (7,0) -- (8,0);
\draw (8,0) -- (9,1.5);
\draw (8,0) -- (9,0.5);
\draw (8,0) -- (9,-0.5);
\draw (8,0) -- (9,-1.5);

\draw (11,0) -- (10,1.5);
\draw (11,0) -- (10,0.5);
\draw (11,0) -- (10,-0.5);
\draw (11,0) -- (10,-1.5);
\draw (11,0) -- (12,0);
\draw (12,3) -- (12,0);
\draw (7,3) -- (7,0);
\draw (7,3) -- (12,3);
\draw (12,1.5) -- (13,1.5);
\draw (6,1.5) -- (7,1.5);

\draw [very thick] (6,1.5) -- (5,3);
\draw [very thick] (6,1.5) -- (5,2);
\draw [very thick] (6,1.5) -- (5,1);
\draw [very thick] (6,1.5) -- (5,0);

\draw (5,3) node [left]{$A_{21}$};
\draw (5,2) node [left]{$A_{23}$};
\draw (5,1) node [left]{$A_{25}$};
\draw (5,0) node [left]{$A_{27}$};

\draw [very thick] (13,1.5) -- (14,3);
\draw [very thick] (13,1.5) -- (14,2);
\draw [very thick] (13,1.5) -- (14,1);
\draw [very thick] (13,1.5) -- (14,0);

\draw (14,3) node [right]{$A_{20}$};
\draw (14,2) node [right]{$A_{22}$};
\draw (14,1) node [right]{$A_{24}$};
\draw (14,0) node [right]{$A_{26}$};

\draw [very thick] (9,1.5) -- (10,1.5);
\draw [very thick] (9,0.5) -- (10,0.5);
\draw [very thick] (9,-0.5) -- (10,-0.5);
\draw [very thick] (9,-1.5) -- (10,-1.5);

\draw (5,3)  node {$\bullet$};
\draw (5,2)  node {$\bullet$};
\draw (5,1)  node {$\bullet$};
\draw (5,0)  node {$\bullet$};

\draw (14,3)  node {$\bullet$};
\draw (14,2)  node {$\bullet$};
\draw (14,1)  node {$\bullet$};
\draw (14,0)  node {$\bullet$};

\draw (6,1.5)  node {$\bullet$};
\draw (13,1.5)  node {$\bullet$};
\draw (11,0)  node {$\bullet$};
\draw (7,0) node {$\bullet$};
\draw (8,0) node {$\bullet$};
\draw (9,1.5) node {$\bullet$};
\draw (9,0.5) node {$\bullet$};
\draw (9,-1.5) node {$\bullet$};
\draw (9,-0.5) node {$\bullet$};
\draw (7,1.5) node {$\bullet$};
\draw (7,3) node {$\bullet$};
\draw (9.5,3) node {$\bullet$};
\draw (12,3) node {$\bullet$};
\draw (12,1.5) node {$\bullet$};
\draw (12,0) node {$\bullet$};

\draw (10,1.5) node {$\bullet$};
\draw (10,0.5) node {$\bullet$};
\draw (10,-1.5) node {$\bullet$};
\draw (10,-0.5) node {$\bullet$};

\draw (11,0) node [above]{$A_{1}$};

\draw (12,0) node [right]{$A_{2}$};
\draw (12,1.5) node [left]{$A_{3}$};
\draw (12,3) node [right]{$A_{5}$};
\draw (9.5,3) node [below]{$A_{6}$};
\draw (7,3) node [left]{$A_{7}$};
\draw (7,1.5) node [right]{$A_{8}$};
\draw (7,0) node [left]{$A_{10}$};
\draw (7.85,0) node [above]{$A_{11}$};
\draw (9.1,1.5) node [below]{$A_{12}$};
\draw (9,0.5) node [below]{$A_{13}$};
\draw (9,-1.5) node [below]{$A_{15}$};
\draw (9,-0.5) node [below]{$A_{14}$};

\draw (9.9,1.5) node [below]{$A_{16}$};
\draw (10,0.5) node [below]{$A_{17}$};
\draw (10,-1.5) node [below]{$A_{19}$};
\draw (10,-0.5) node [below]{$A_{18}$};
    		\end{tikzpicture}}
\caption{Dual graph for $\mathrm{NS}(\mathcal{X}) = H\oplus E_8 \oplus A_1^{\oplus 4}$}
 \label{fig:pic14}
\end{figure}
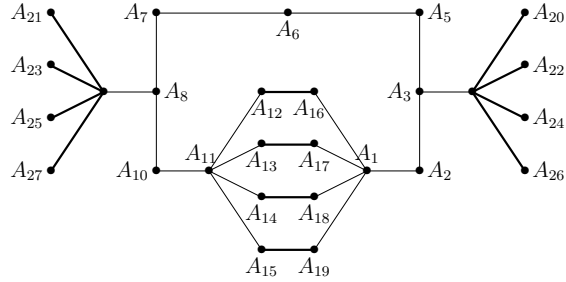
\begin{figure}
  	\centering
	\scalebox{0.7}{%
    		\begin{tikzpicture}
       		\draw (0,0) -- (6,0);
\draw (1,1.5) -- (5,1.5);
\draw (0,0) -- (5,1.5);
\draw (1,1.5) -- (6,0);
\draw [very thick] (2,1.5) -- (4,1.5);
\draw [very thick] (2,2.5) -- (4,2.5);
\draw [very thick] (2,3.5) -- (4,3.5);
\draw [ultra thick] (0,4.5) -- (6,4.5);
\draw [very thick] (0,4.5) -- (0,2.5);
\draw [very thick] (6,4.5) -- (6,2.5);
\draw (0,2.5) -- (0,0);
\draw (6,2.5) -- (6,0);
\draw (0,4.5) -- (2,1.5);
\draw (0,4.5) -- (2,2.5);
\draw (0,4.5) -- (2,3.5);
\draw (6,4.5) -- (4,1.5);
\draw (6,4.5) -- (4,2.5);
\draw (6,4.5) -- (4,3.5);
\draw (1,1.5) -- (2,1.5);
\draw (1,1.5) -- (2,2.5);
\draw (1,1.5) -- (2,3.5);
\draw (5,1.5) -- (4,1.5);
\draw (5,1.5) -- (4,2.5);
\draw (5,1.5) -- (4,3.5);

\draw (0,0) node {$\bullet$};
\draw (1,0) node {$\bullet$};
\draw (2,0) node {$\bullet$};
\draw (3,0) node {$\bullet$};
\draw (4,0) node {$\bullet$};
\draw (5,0) node {$\bullet$};
\draw (6,0) node {$\bullet$};
\draw (1.5,0.45) node {$\bullet$};
\draw (4.5,0.45) node {$\bullet$};
\draw (0,2.5)  node {$\bullet$};
\draw (6,2.5)  node {$\bullet$};
\draw (0,4.5)  node {$\bullet$};
\draw (2,1.5) node {$\bullet$};
\draw (2,2.5) node {$\bullet$};
\draw (2,3.5) node {$\bullet$};
\draw (6,4.5)  node {$\bullet$};
\draw (4,1.5) node {$\bullet$};
\draw (4,2.5) node {$\bullet$};
\draw (4,3.5) node {$\bullet$};
\draw (1,1.5)  node {$\bullet$};
\draw (5,1.5)  node {$\bullet$};

\draw (0,0) node [below]{$A_{10}$};
\draw (1,0) node [below]{$A_{9}$};
\draw (2,0) node [below]{$A_{8}$};
\draw (3,0) node [below]{$A_{7}$};
\draw (4,0) node [below]{$A_{6}$};
\draw (5,0) node [below]{$A_{5}$};
\draw (6,0) node [below]{$A_{3}$};
\draw (0,2.5) node [left]{$A_{11}$}; 
\draw (6,2.5)  node [right]{$A_{4}$};
\draw (0,4.5)  node [left]{$A_{20}$};
\draw (2,1.5) node [above]{$A_{17}$};
\draw (2,2.5) node [above]{$A_{18}$};
\draw (2,3.5) node [above]{$A_{19}$};
\draw (6,4.5)  node [right]{$A_{21}$};
\draw (4-0.15,1.5) node [above]{$A_{14}$};
\draw (4-0.15,2.5) node [above]{$A_{15}$};
\draw (4-0.15,3.5) node [above]{$A_{16}$};
\draw (1,1.5)  node [left]{$A_{1}$};
\draw (5,1.5)  node [right]{$A_{13}$};
\draw (1.5,0.45) node [above]{$A_{12}$};
\draw (4.5,0.45) node [above]{$A_{2}$};

\draw (3,4.5) node [above]{\small $6$};
    		\end{tikzpicture}}
\caption{Dual graph for $\mathrm{NS}(\mathcal{X}) = H \oplus E_8  \oplus D_4 \oplus A_1)$}
 \label{fig:pic15}
\end{figure}
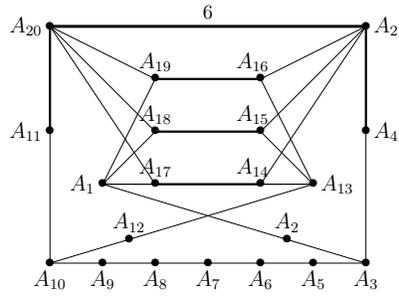
\begin{figure}
  	\centering
	\scalebox{0.7}{%
    		\begin{tikzpicture}
       		\draw [very thick] (0,1) -- (12,1);
\draw [very thick] (0,-1) -- (12,-1);
\draw (0,0) -- (12,0);
\draw (0,1) -- (0,-1);
\draw (12,1) -- (12,-1);

\draw (2,-0.8) -- (2,0); 
\draw (10,-0.8) -- (10,0);

\draw (0,1) node {$\bullet$};
\draw (0,-1) node {$\bullet$};
\draw (12,1) node {$\bullet$};
\draw (12,-1) node {$\bullet$};

\draw (0,0) node {$\bullet$};
\draw (1,0) node {$\bullet$};
\draw (2,0) node {$\bullet$};
\draw (3,0) node {$\bullet$};
\draw (4,0) node {$\bullet$};
\draw (5,0) node {$\bullet$};
\draw (6,0) node {$\bullet$};
\draw (7,0) node {$\bullet$};
\draw (8,0) node {$\bullet$};
\draw (9,0) node {$\bullet$};
\draw (10,0) node {$\bullet$};
\draw (11,0) node {$\bullet$};
\draw (12,0) node {$\bullet$};
\draw (2,-0.8) node {$\bullet$};
\draw (10,-0.8) node {$\bullet$};

\draw (0,0) node [left]{$A_{1}$};
\draw (1,0) node [above]{$A_{2}$};
\draw (2,0) node [above]{$A_{3}$};
\draw (3,0) node [above]{$A_{5}$};
\draw (4,0) node [above]{$A_{6}$};
\draw (5,0) node [above]{$A_{7}$};
\draw (6,0) node [above]{$A_{8}$};
\draw (7,0) node [above]{$A_{9}$};
\draw (8,0) node [above]{$A_{10}$};
\draw (9,0) node [above]{$A_{11}$};
\draw (10,0) node [above]{$A_{12}$};
\draw (11,0) node [above]{$A_{14}$};
\draw (12,0) node [right]{$A_{15}$};

\draw (0,1) node [left]{$A_{19}$};
\draw (0,-1) node [left]{$A_{17}$};
\draw (12,1) node [right]{$A_{18}$};
\draw (12,-1) node [right]{$A_{16}$};
\draw (2,-0.7) node [right]{$A_{4}$};
\draw (10,-0.7) node [right]{$A_{13}$};
    		\end{tikzpicture}}
\caption{Dual graph for $\mathrm{NS}(\mathcal{X}) = H \oplus E_8  \oplus D_6$}
\label{fig:pic16}
\end{figure}
\begin{figure}
  	\centering
	\scalebox{0.7}{%
    		\begin{tikzpicture}
       		\draw [very thick] (3,-2) -- (7,-2);
\draw (0,-2) -- (10,-2);
\draw (0,0) -- (10,0);
\draw (0,1) -- (0,-2);
\draw (10,1) -- (10,-2);

\draw (0,0) node {$\bullet$};
\draw (1,0) node {$\bullet$};
\draw (2,0) node {$\bullet$};
\draw (3,0) node {$\bullet$};
\draw (4,0) node {$\bullet$};
\draw (5,0) node {$\bullet$};
\draw (6,0) node {$\bullet$};
\draw (7,0) node {$\bullet$};
\draw (8,0) node {$\bullet$};
\draw (9,0) node {$\bullet$};
\draw (10,0) node {$\bullet$};

\draw (0,1) node {$\bullet$};
\draw (0,-1) node {$\bullet$};
\draw (0,-2) node {$\bullet$};
\draw (3,-2) node {$\bullet$};

\draw (10,1) node {$\bullet$};
\draw (10,-1) node {$\bullet$};
\draw (10,-2) node {$\bullet$};
\draw (7,-2) node {$\bullet$};

\draw (0,0) node [left]{$A_{2}$};
\draw (1,0) node [above]{$A_{6}$};
\draw (2,0) node [above]{$A_{7}$};
\draw (3,0) node [above]{$A_{8}$};
\draw (4,0) node [above]{$A_{9}$};
\draw (5,0) node [above]{$A_{10}$};
\draw (6,0) node [above]{$A_{11}$};
\draw (7,0) node [above]{$A_{12}$};
\draw (8,0) node [above]{$A_{13}$};
\draw (9,0) node [above]{$A_{14}$};

\draw (0,1) node [left]{$A_{1}$};
\draw (0,-1) node [left]{$A_{3}$};
\draw (0,-2) node [left]{$A_{4}$};
\draw (3,-2) node [above]{$A_{5}$};

\draw (10,0) node [right]{$A_{15}$};
\draw (10,1) node [right]{$A_{18}$};
\draw (10,-1) node [right]{$A_{19}$}; 
\draw (10,-2) node [right]{$A_{16}$}; 
\draw (7,-2) node [above]{$A_{17}$};
    		\end{tikzpicture}}
\caption{Dual graph for $\mathrm{NS}(\mathcal{X}) = H \oplus E_8 \oplus E_7$}
\label{fig:pic17}
\end{figure}
\begin{figure}
  	\centering
	\scalebox{0.8}{%
    		\begin{tikzpicture}
\draw (0,0) -- (16*0.8,0);
\draw (1.6,0) -- (1.6,-0.8);
\draw (0.8*14,0) -- (0.8*14,-0.8);

\draw (0,0) node {$\bullet$};
\draw (0.8,0) node {$\bullet$};
\draw (2*0.8,0) node {$\bullet$};
\draw (3*0.8,0) node {$\bullet$};
\draw (4*0.8,0) node {$\bullet$};
\draw (5*0.8,0) node {$\bullet$};
\draw (6*0.8,0) node {$\bullet$};
\draw (7*0.8,0) node {$\bullet$};
\draw (8*0.8,0) node {$\bullet$};
\draw (9*0.8,0) node {$\bullet$};
\draw (10*0.8,0) node {$\bullet$};
\draw (11*0.8,0) node {$\bullet$};
\draw (12*0.8,0) node {$\bullet$};
\draw (13*0.8,0) node {$\bullet$};
\draw (14*0.8,0) node {$\bullet$};
\draw (15*0.8,0) node {$\bullet$};
\draw (16*0.8,0) node {$\bullet$};
\draw (0.8*2,-0.8) node {$\bullet$};
\draw (0.8*14,-0.8) node {$\bullet$};

\draw (0,0) node [above]{$A_{1}$};
\draw (0.8,0) node [above]{$A_{2}$};
\draw (2*0.8,0) node [above]{$A_{3}$};
\draw (3*0.8,0) node [above]{$A_{5}$};
\draw (4*0.8,0) node [above]{$A_{6}$};
\draw (5*0.8,0) node [above]{$A_{7}$};
\draw (6*0.8,0) node [above]{$A_{8}$};
\draw (7*0.8,0) node [above]{$A_{19}$};
\draw (8*0.8,0) node [above]{$A_{9}$};
\draw (9*0.8,0) node [above]{$A_{10}$};
\draw (10*0.8,0) node [above]{$A_{11}$};
\draw (11*0.8,0) node [above]{$A_{12}$};
\draw (12*0.8,0) node [above]{$A_{13}$};
\draw (13*0.8,0) node [above]{$A_{14}$};
\draw (14*0.8,0) node [above]{$A_{15}$};
\draw (15*0.8,0) node [above]{$A_{17}$};
\draw (16*0.8,0) node [above]{$A_{18}$};
\draw (0.8*2,-0.8) node [left]{$A_{4}$};
\draw (0.8*14,-0.8) node [right]{$A_{16}$};
    		\end{tikzpicture}}
\caption{Dual graph for $\mathrm{NS}(\mathcal{X}) = H \oplus E_8  \oplus E_8$}
\label{fig:pic18}
\end{figure}
\section{Conclusions}
We have proved that every K3 surface with automorphism group $(\mathbb{Z}/2\mathbb{Z})^2$ admits a birational model as a double sextic surface. This was the central part of Theorem~\ref{thm1_intro}, stated in the introduction. In Table~\ref{tab:proof_thm1} we indicate where each part of Theorem~\ref{thm1_intro} has been proved in this article, depending on the rank of the corresponding lattice. Furthermore, the constructed birational model $\mathcal{S}$ is canonical for $\rho_L >10$. This was stated as Proposition~\ref{prop:canonical_intro} in the introduction and proved as part of Proposition~\ref{prop:double_surface_S_rank11ff}. For Picard number greater than 9, the considered K3 surfaces also possess a second birational model as quartic projective hypersurface $\mathcal{K}$, generalizing the Inose quartic.  The polarizing divisor of this quartic was computed explicitly. This was the central part of Theorem~\ref{thm2_intro},  stated in the introduction.  In Table~\ref{tab:proof_thm2} we indicate where every statement contained in Theorem~\ref{thm2_intro} has been proved in this article, depending on the rank of the corresponding lattice.  The birational model as quartic hypersurface $\mathcal{K}$ is closely related to that of a second double sextic, denoted by $\mathcal{S}'$, whose branch locus is the union of a conic and a plane quartic curve. Explicit equations for the constructed surfaces were also provided for all lattices, with equations as indicated in Table~\ref{tab:surfaces}.
\begin{table}
\begin{tabular}{c|c||l|l|l}
$\rho_L$ 	& $(\ell_L, \delta_L)$	& Thm.~\ref{thm1_intro}~(1) & Thm.~\ref{thm1_intro}~(2) & Thm.~\ref{thm1_intro}~(3) \\
\hline 
\hline
$9$ 		& $(9,1)$ 	& Prop.~\ref{prop1} 		& Prop.~\ref{prop0} 					& Prop.~\ref{prop2} \\
\hline
$10$ 	& $(6,0)$ 	& Prop.~\ref{prop:rank10} & Prop.~\ref{prop:rank10}				& Prop.~\ref{prop:symmetry_rank10a}\\
\hdashline
$10$ 	& $(8,1)$ 	& Cor.~\ref{cor:rank10b}	& Lemma~\ref{lem:sings_on_S_rank10} 	& Prop.~\ref{prop:symmetry_rank10b}\\
\hline
$18-k$ & $(k, \delta_k)$ 	& Prop.~\ref{prop:double_surface_S_rank11ff} & Prop.~\ref{prop:double_surface_S_rank11ff}~(1), (2) & Thm.~\ref{thm_rank_ge11}~(4)\\[-0.1em]
\multicolumn{5}{l}{{\small $\quad k =0, \dots, 7$, with $\delta_k =1$ for $k= 1, \dots, 7$ and $\delta_k =0$ for $k= 0, 4$}}  \\
\hline
\end{tabular}
\captionsetup{justification=centering}
\caption{Proof of Theorem~\ref{thm1_intro}}
\label{tab:proof_thm1}
\end{table}
\begin{table}
\begin{tabular}{c|c||l|l|l}
$\rho_L$ 	& $(\ell_L, \delta_L)$ & Thm.~\ref{thm2_intro}~(1) & Thm.~\ref{thm2_intro}~(2) & Thm.~\ref{thm2_intro}~(3) \\
\hline 
\hline
$9$ 		& $(7,1)$	& Prop.~\ref{prop:K_rank9} 									& Lemma~\ref{lem:sing_of_K_rank9} 	& n/a\\
\hline
$10$ 	& $(8,1)$ 	& Prop.~\ref{prop:double_surface_Sp_10}, Lemma~\ref{lem:birational}	& Lemma~\ref{lem:sings_K_rank10} 		& Prop.~\ref{prop_divisor_rank10}\\
\hline
$18-k$ & $(k, \delta_k)$ 	& \multicolumn{2}{|c|}{Cor.~\ref{prop:regular}} 						& Thm.~\ref{thm_rank_ge11}~(3)\\[-0.1em]
\multicolumn{5}{l}{{\small $\quad k =0, \dots, 7$, with $\delta_k =1$ for $k= 1, \dots, 7$ and $\delta_k =0$ for $k= 0, 4$}}  \\
\hline
\end{tabular}
\captionsetup{justification=centering}
\caption{Proof of Theorem~\ref{thm2_intro}}
\label{tab:proof_thm2}
\end{table}
\begin{table}
\begin{tabular}{c|c||l|l|l}
$\rho_L$ 	& $(\ell_L, \delta_L)$ & $\mathcal{S}$ & $\mathcal{S}'$ & $\mathcal{K}$ \\
\hline 
\hline
$9$ 	& $(9,1)$  & Eq.~(\ref{eqn:canonical_sextic_rank9star}) 	& n/a														& n/a \\
\hdashline
$9$ 	& $(7,1)$	& n/a									& Eq.~(\ref{eqn:double-sextics_prime_general_intro}) (general $C, Q$)	& Eq.~(\ref{eqn:quartic_general_intro})  (general $C, Q$) \\
\hline
$10$ & $(6,0)$ 	& Eq.~(\ref{eqn:double_sextic_rank10}) 		& n/a														& n/a \\
\hdashline
$10$ & $(8,1)$ 	& Eq.~(\ref{eqn:canonical_sextic_rank10star})	& Eq.~(\ref{eqn:canonical_sextic_prime_rank10star})				& Eq.~(\ref{eqn:canonical_quartic_rank10star})  \\
\hline
$18-k$ & $(k, \delta_k)$ & Eq.~(\ref{eqn:double-sextics_general_rang_ge11})	& Eq.~(\ref{eqn:double-sextics_prime_general_rank_ge11})	& Eq.~(\ref{eqn:quartic_general_rank_ge11})\\[-0.1em]
\multicolumn{5}{l}{{\small $\quad k =0, \dots, 7$, with $\delta_k =1$ for $k= 1, \dots, 7$ and $\delta_k =0$ for $k= 0, 4$}}  \\
\hline
\end{tabular}
\captionsetup{justification=centering}
\caption{Normal forms for the surfaces $\mathcal{S}$,  $\mathcal{S}'$,  $\mathcal{K}$}
\label{tab:surfaces}
\end{table}
\bibliographystyle{amsplain}
\bibliography{references1, references2}{}
\end{document}